\documentclass[12pt]{amsart}
\usepackage{amsfonts}
\usepackage{amsmath}
\usepackage{amssymb}
\usepackage[margin=1in]{geometry}
\usepackage{tikz}
\usepackage{amsthm}   							
\usepackage{mathtools}							
\usepackage{commath}

\usepackage[english]{babel}						
\usepackage{graphicx,xcolor}

\makeatletter
\@namedef{subjclassname@2020}{\textup{2020} Mathematics Subject Classification}
\makeatother

\newtheorem{theorem}{Theorem}[section]
\newtheorem{corollary}[theorem]{Corollary}
\newtheorem{lemma}[theorem]{Lemma}

\theoremstyle{definition}

\newtheorem{remark}[theorem]{Remark}

\numberwithin{equation}{section}

\newcommand{\N}{\mathbb{N}}
\newcommand{\Z}{\mathbb{Z}}

\newcommand{\R}{\mathbb{R}}
\newcommand{\C}{\mathbb{C}}

\renewcommand{\Re}{\operatorname{Re}}
\renewcommand{\Im}{\operatorname{Im}}

\newcommand{\I}{\mathrm{i}}
\newcommand{\e}{\mathrm{e}}

\newcommand{\eps}{\varepsilon}

\newcommand{\MP}{\mathcal{P}}
\newcommand{\MN}{\mathcal{N}}

\newcommand{\SSS}{\mathcal{S}}
\newcommand{\RR}{\mathcal{R}}
\newcommand{\NP}{\mathcal{N}_{\mathcal P}}
\newcommand{\MM}{\mathcal{M}}

\newcommand{\psip}{\psi_{\mathcal{P}}}

\newcommand{\Pip}{\Pi_{\mathcal{P}}}
\newcommand{\pip}{\pi_{\mathcal{P}}}

\newcommand{\zs}{\zeta_{\mathcal{P}}(s)}

\newcommand{\de}{\delta}
\newcommand{\si}{\sigma}

\newcommand{\om}{\omega}

\definecolor{green}{rgb}{0,0.8,0}

\definecolor{darkgreen}{rgb}{0,0.6,0.2}

\DeclareMathOperator{\Li}{Li}
\DeclareMathOperator{\li}{li}
\DeclareMathOperator*{\res}{Res}

\begin{document}

\title{Some examples of well-behaved Beurling number systems}

\author[F. Broucke]{Frederik Broucke}
\thanks{F. Broucke was supported by a postdoctoral fellowship (grant number 12ZZH23N) of the Research Foundation -- Flanders}

\author[G.~Debruyne]{Gregory Debruyne}
\thanks{G. Debruyne was supported by a postdoctoral fellowship from Ghent University (grant number 01P13022)}

\address{Department of Mathematics: Analysis, Logic and Discrete Mathematics\\ Ghent University\\ Krijgslaan 281\\ 9000 Gent\\ Belgium}

\email{fabrouck.broucke@ugent.be}
\email{gregory.debruyne@ugent.be}

\author[Sz. Gy. R\'ev\'esz]{Szil\'ard Gy.\ R\'ev\'esz}
\address{HUN-REN Alfréd Rényi Institute of Mathematics, Reáltanoda utca 13-15 Budapest, 1053 Hungary} \thanks{The work of Sz. Gy. R\'ev\'esz was supported in part by Hungarian National Research, Development and Innovation Fund projects \# K-119528, \# K-146387 and \# K-147153}

\email{revesz.szilard@renyi.hu}



\begin{abstract}
We investigate the existence of well-behaved Beurling number systems, which are systems of Beurling generalized primes and integers which admit a power saving in the error term of both their prime and integer-counting function. Concretely, we search for so-called $[\alpha,\beta]$-systems, where $\alpha$ and $\beta$ are connected to the optimal power saving in the prime and integer-counting functions. It is known that every $[\alpha,\beta]$-system satisfies $\max\{\alpha,\beta\}\ge1/2$. In this paper we show there are $[\alpha,\beta]$-systems for each $\alpha \in [0,1)$ and $\beta \in [1/2, 1)$. Assuming the Riemann hypothesis, we also construct certain families of $[\alpha,\beta]$-systems with $\beta<1/2$.
\end{abstract}
\maketitle
\section{Introduction}
The \emph{Riemann hypothesis} (RH) is one of the central open problems in number theory today, stating that every non-trivial zero of $\zeta(s)$ must lie on the critical line $\Re s = 1/2$. It is intimately connected to the distribution of prime numbers since an equivalent formulation of RH is that the prime counting function satisfies $\pi(x) = \Li(x) + O_{\eps}(x^{1/2 + \eps})$, for each $\eps > 0$, where $\Li$ stands for the logarithmic integral, defined here as $\Li(x)\coloneqq \int_{1}^{x}\frac{1-u^{-1}}{\log u}\dif u$. In this work we investigate whether it is possible in general multiplicative structures to reconcile \emph{quasi-RH}-type estimates on the primes with strong estimates on the integers. We will carry out this study in the context of Beurling generalized prime numbers.

A \emph{Beurling generalized number system} $(\mathcal{P}, \mathcal{N})$ consists of a sequence of \emph{generalized primes} $\mathcal{P} = (p_{j})_{j}$ with $1 < p_{1} \leq  p_{2} \leq \dots $ and $p_{j}\to \infty$, and the \emph{generalized integers} $\mathcal{N}$ which are formed by taking the multiplicative semigroup generated by the generalized primes and $1$. Arranging the generalized integers in non-decreasing order\footnote{We consider generalized integers to be different if their generalized prime factorization is different, even if they share the same numerical value.}, one obtains the sequence $1 = n_{0} < n_{1}= p_{1} \leq n_{2} \leq \dots$.  One may associate to this system many familiar number-theoretic functions, such as
\[
	\pi_{\mathcal{P}}(x) = \sum_{{p_j} \leq x} 1, \quad N_{\mathcal{P}}(x) = \sum_{n_j \leq x} 1, \quad \zeta_{\mathcal{P}}(s) = \sum_{n_j} \frac{1}{{n}^{s}_j}.
\]
The relationship between these three functions has been the subject of extensive research over the last century. We refer to the monograph \cite{DiamondZhangbook} for a detailed account on the theory of Beurling systems.

This paper considers Beurling number systems $(\MP,\MN)$ whose integers have a density and whose primes satisfy the Prime Number Theorem. These are systems $(\MP, \MN)$ with $N_{\MP}(x) \sim ax$ for some $a>0$, and $\psi_{\MP}(x)\sim x$. Here $\psi_{\MP}(x) \coloneqq \sum_{p_{j}^{k}\le x}\log p_{j}$ denotes the Chebyshev function of the Beurling system. More specifically we investigate the existence of systems which admit a power saving in the error terms for both these asymptotic relations. Let $\alpha,\beta\in [0,1)$. An \emph{$[\alpha,\beta]$-system}, a notion introduced by Hilberdink \cite{Hilberdink2005}, is a Beurling generalized number system $(\MP, \MN)$ for which there exists $a > 0$ such that both
\begin{align}
\label{eq: well-behaved primes}
	\psi_{\MP}(x) 	&= x + O_{\eps}(x^{\alpha+\eps}), \\
\label{eq: well-behaved integers}
	N_{\MP}(x)	&= a x + O_{\eps}(x^{\beta+\eps})
\end{align}
hold\footnote{In other words,
\[
\alpha=\limsup_{x\to \infty} \frac{\log|\psi_{\MP}(x)-x|}{\log x}\qquad \textrm{and}  \qquad \beta=\limsup_{x\to \infty} \frac{\log|N_{\MP}(x)-ax|}{\log x}.
\]
} for every $\eps>0$, but for no $\eps<0$.

The formula \eqref{eq: well-behaved primes} being valid for each $\eps > 0$ is equivalent to $\Pi_{\MP}(x) = \Li(x) + O_\eps(x^{\alpha + \eps})$ and $\pi_{\MP}(x) = \li(x) + O_\eps(x^{\alpha + \eps})$. Here $\Pi_{\MP}(x)$ stands for the Riemann prime counting function $\Pi_{\MP}(x) \coloneqq \sum_{k = 1}^{\infty}\frac{1}{k} \pi_{\MP}(x^{1/k})$ and $\li$, which is a modification of the logarithmic integral, is defined as $\li(x) \coloneqq \sum_{k=1}^{\infty}\frac{\mu(k)}{k}\Li(x^{1/k})$.
In \cite{Hilberdink2005}, Hilberdink showed that for any $[\alpha, \beta]$-system, we have $\max\{\alpha, \beta\} \ge 1/2$. This result can be interpreted as a type of \emph{uncertainty principle}. If one wishes strong control on the primes then one must relinquish the control of the integers and vice-versa. One may conjecture that for all $\alpha,\beta$ with $\max\{\alpha, \beta\}\ge1/2$, there must exist a corresponding $[\alpha,\beta]$-system. The main goal of this paper is to establish this conjecture in certain ranges for $\alpha$ and $\beta$.

If \eqref{eq: well-behaved primes} holds for every $\eps>0$, we will say that \emph{the primes (of the system $\mathcal{P}$) are $\alpha$-well-behaved} and omit $\alpha$ if it is clear from the context. Similarly the expression \eqref{eq: well-behaved integers} for every $\eps>0$ shall be referred to as \emph{the integers are $\beta$-well-behaved}. Regarding the relationship between both counting functions, it is known that having $\beta$-well-behaved integers implies that $\psi_{\MP}(x) = x + O\bigl((x\exp(-c\sqrt{\log x})\bigr)$, and having $\alpha$-well-behaved primes implies that $N(x) = ax + O\bigl(x\exp(-c'\sqrt{\log x\log\log x})\bigr)$, for some $a>0$ and $c, c'>0$ depending on $\beta$ and $\alpha$ respectively. In general, these exponential error terms are optimal, as was shown in \cite{DiamondMontgomeryVorhauer} for the first implication and in \cite{BDV2020} for the second one (see also \cite{BDV2022}).

If only \eqref{eq: well-behaved primes} holds for all $\eps>0$ and no $\eps<0$, but the integers are not $\beta$-well-behaved for any $\beta <1$, the system is called an $[\alpha,1]$-system. Analogously a system\footnote{We refrain here from defining $[1,1]$-systems as there are multiple possibilities for its definition. One may for instance say that the primes are $1$-well-behaved if the PNT holds, Chebyshev estimates hold or even only $\psi(x) \ll_\eps x^{1+\eps}$. As mentioned before, if $\min\{\alpha,\beta\} < 1$, then the Beurling system must admit both the PNT and density, so that the primes, resp.\ integers are then always $1$-well-behaved, whatever \emph{sensible} definition is chosen.} is called $[1,\beta]$ if \eqref{eq: well-behaved integers} holds for every $\eps>0$ and no $\eps<0$, but the primes are not $\alpha$-well-behaved for any $\alpha < 1$. We say that a Beurling system is well-behaved if both $\alpha$ and $\beta$ are less than $1$.

\medskip

When Hilberdink \cite{Hilberdink2005} introduced the notion of $[\alpha,\beta]$-systems, the only known examples were conjectural, namely the rational primes and integers $(\mathbb{P},\N)$, which form a $[1/2,0]$-system under the Riemann hypothesis, and more generally the systems associated to the prime and integral ideals of algebraic number fields, which form $[1/2, \beta]$-systems ($\beta<1$ depending on the number field) under the extended Riemann hypothesis. In \cite{Hilberdink2005}, Hilberdink mentions that Montgomery told him he might be able to construct certain $[\alpha,\beta]$-systems using a probabilistic method from \cite{DiamondMontgomeryVorhauer}. The first unconditional example of a well-behaved system was established by Zhang \cite{Zhang2007}, who used this probabilistic method to show the existence of a system satisfying $\psi_{\MP}(x) = x + O_{\eps}(x^{1/2+\eps})$ and $N_{\MP}(x) = ax + O_{\eps}(x^{1/2+\eps})$ for all $\eps>0$. Due to the probabilistic nature of the method, no precise value of $\alpha$ and $\beta$ could be determined.

Recently, Vindas and the first author \cite{BrouckeVindas} found a refinement of the probabilistic method which allows a much better control on the randomly constructed primes. In that paper, the existence of $[0,1/2]$-systems was shown.

Up to now, no general classes of Beurling number systems were constructed to demonstrate that for general pairs of parameters such systems may exist. We address this issue in the present paper. The first main theorem of this work is

\begin{theorem}
\label{th: case I and II}
For any $\alpha \in [0,1)$ and $\beta \in [1/2, 1)$ there exists an $[\alpha, \beta]$-system.
\end{theorem}

It is also possible to construct $[\alpha, 1]$-systems for any $0\leq \alpha < 1$ and $[1,\beta]$-systems for any $1/2 \leq  \beta < 1$, but the examples that we shall present here do not fully go through in these cases. An example of a $[0,1]$-system can be found in \cite{BDV2022} and that of a $[1,\beta_0]$-system\footnote{Most likely the value of $\beta_0$ equals $1/2$, but in principle it is still possible that $\beta_0$ could be smaller.} is in \cite[Ch.\ 17]{DiamondZhangbook} (which is based on the papers \cite{DiamondMontgomeryVorhauer} and \cite{Zhang2007}) for some $\beta_0 \leq 1/2$. The construction for arbitrary $[\alpha,1]$- and $[1,\beta]$-systems with $\beta\ge1/2$ only requires some minor modifications to these examples which we shall briefly discuss. It should be noted that in both these works, much more precise estimates on $\psi_{\MP}(x)$ and $N_{\MP}(x)$ than what is needed to be an $[\alpha,1]$- or $[1,\beta]$-system were obtained.

The construction of $[\alpha,\beta]$-systems will be done in two stages. First, a generalized number system in the \emph{extended sense} will be constructed which satisfies the requirements. We define a number system in the extended sense as a pair of non-decreasing right-continuous functions $(\Pi(x), N(x))$, supported on $[1,\infty)$, which satisfy $\Pi(1)=0$, $N(1)=1$, and
\[
	\zeta(s) \coloneqq \int_{1^{-}}^{\infty}x^{-s}\dif N(x) = \exp \int_{1}^{\infty}x^{-s}\dif \Pi(x)
\]
(or equivalently, $\dif N(x) = \exp^{\ast}(\dif \Pi(x))$, where the exponential is taken with respect to the multiplicative convolution of measures \cite[Ch.\ 2--3]{DiamondZhangbook}).

Next, we show that this number system in the extended sense can be suitably approximated by an actual \emph{discrete} system $(\MP, \MN)$. The key technical tool for this is the following theorem (\cite[Th.\ 1.2]{BrouckeVindas}).
\begin{theorem}
\label{th: discretization}
Let $F$ be a non-decreasing right-continuous function tending to $\infty$, with $F(1) = 0$ and satisfying the Chebyshev upper bound $F(x)\ll x/\log x$. Then there exists a sequence of generalized primes $\MP=(p_{j})_{j=1}^{\infty}$ such that $\abs{\pi_{\MP}(x)-F(x)}\le 2$ and such that for any real $t$ and any $x\ge1$
\begin{equation}
\label{eq: bound exp sums}
	\abs[3]{\sum_{p_{j}\le x}p_{j}^{-\I t} - \int_{1}^{x}u^{-\I t}\dif F(u)} \ll \sqrt{x} + \sqrt{\frac{x\log(\,\abs{t}+1)}{\log(x+1)}}.
\end{equation}
\end{theorem}

\medskip

We cannot apply the above method to construct $[\alpha,\beta]$-systems with $\beta<1/2$, since Theorem \ref{th: discretization} can only control the distribution of the integers up to an error of size $O(x^{1/2+\eps})$, as we shall later see (this is basically due to the presence of the factor $\sqrt{x}$ in the right hand side of \eqref{eq: bound exp sums}). In fact, any method for approximating systems in the extended sense by discrete systems $(\MP, \MN)$ which yields $O(x^{\theta})$ control on either $\Pi_{\MP}(x)$ or $N_{\MP}(x)$, where $\theta<1/2$, should have an uncertainty of size at least $x^{1/2-\eps}$ for every $\eps>0$ on the other counting function. If not, then approximating the extended system $(\Pi(x),N(x)) = (\Li(x), x)$, for which $\zeta(s) = s/(s-1)$, would yield an $[\alpha,\beta]$-system with $\max\{\alpha,\beta\} < 1/2$, contradicting Hilberdink's result.

Instead of creating new systems from scratch, we will modify the classical system of rational primes and integers $(\mathbb{P}, \N)$. Assuming the Riemann hypothesis, this is a $[1/2, 0]$-system. We will introduce certain perturbations of $(\mathbb{P}, \N)$ giving rise to new families of $[\alpha,\beta]$-systems with $\beta<1/2$.
\begin{theorem} \label{wbs: thrh}
Assume RH. Then there exists a $[1/2, \beta]$-system for each $0 \leq \beta < 1/2$ and an $[\alpha,\beta]$-system for $1/2 < \alpha < 2/3$ and $2\alpha/(\alpha + 2) \leq \beta < 1/2$.
\end{theorem}

After a short preliminary section where we recall some basic notions of the logarithmic integral, we start with the proof of Theorem \ref{th: case I and II}. The two cases $1/2\le \beta\le \alpha$ and $\alpha<\beta$ require different approaches. In the first case, we prescribe certain zeros and poles for the Beurling zeta function, while in the second case, we prescribe a certain kind of extreme growth. These two cases are covered in Section \ref{sec:betalessalpha} and \ref{sec:alphalessbeta} respectively. In Section \ref{sec:betalesshalf} we prove Theorem \ref{wbs: thrh} by perturbing the system $(\mathbb{P}, \N)$. A schematic overview of the covered cases can be found in Figure \ref{schematic overview}.

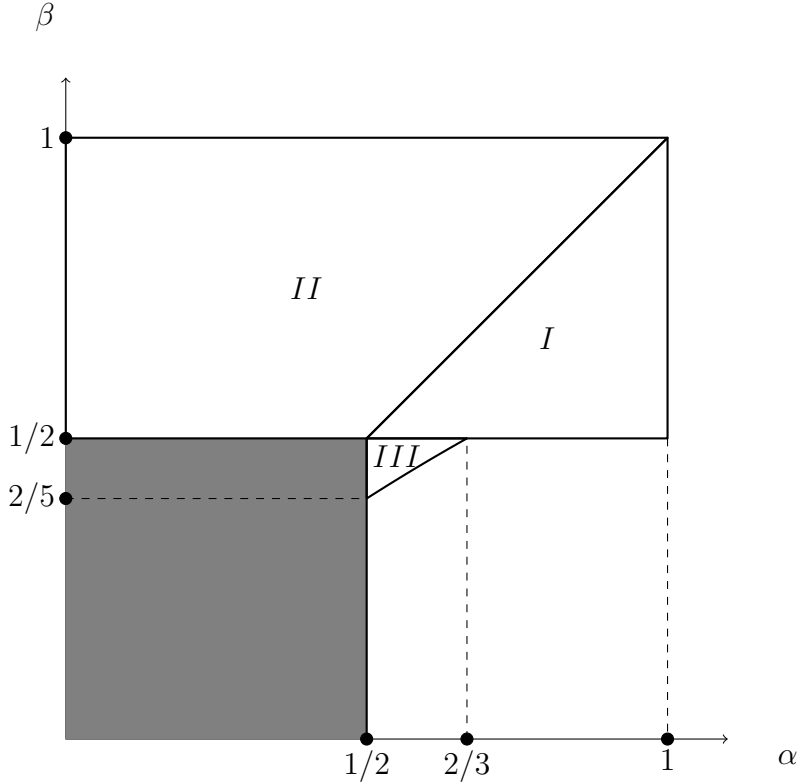
\begin{figure}\label{schematic overview}
\begin{tikzpicture}[scale=8, line join=round]
\draw [<->] (0,1.1) -- (0,0) -- (1.1,0);
\node[left] at (0, 1.2) {$\beta$};
\node[below] at (1.2, 0) {$\alpha$};

\draw[gray, fill=gray] (0,0) -- (1/2,0) -- (1/2,1/2) -- (0,1/2) -- (0,0);
\draw[thick, domain= (1/2:2/3)] plot(\x, {2*\x/(\x+2)}) -- (1/2,1/2) -- (1/2, 2/5);
\draw[thick] (0,1) -- (1,1) -- (1/2,1/2) -- (0,1/2) -- (0,1);
\draw[thick] (1/2, 1/2) -- (1,1) -- (1,1/2) -- (1/2,1/2);
\draw[thick] (1/2,0) -- (1/2,1/2);

\node at (2/5,3/4) {$II$};
\node at (4/5,2/3) {$I$};
\node at (0.55, 0.47) {$III$};

\draw[fill] (0,1) circle [radius=0.01];
\node[left] at (0,1) {$1$};
\draw[fill] (0,1/2) circle [radius=0.01];
\node[left] at (0,1/2) {$1/2$};
\draw[fill] (1,0) circle [radius=0.01];
\node[below] at (1,0) {$1$};
\draw[fill] (1/2,0) circle [radius=0.01];
\node[below] at (1/2,0) {$1/2$};
\draw[fill] (2/3,0) circle [radius=0.01];
\node[below] at (2/3,0) {$2/3$};
\draw[fill] (0,2/5) circle [radius=0.01];
\node[left] at (0,2/5) {$2/5$};

\draw[dashed] (1,1/2) -- (1,0);
\draw[dashed] (2/3,0)--(2/3, 1/2);
\draw[dashed] (0,2/5)--(1/2,2/5);
\end{tikzpicture}
\caption{Schematic overview of the constructed well-behaved systems. The gray region is not possible due to Hilberdink's result. The systems constructed in Section \ref{sec:betalessalpha} correspond to region $I$, those from Section \ref{sec:alphalessbeta} to region $II$, and those constructed in Section \ref{sec:betalesshalf} under RH correspond to region $III$.}
\end{figure}

We recall some classical notation that we will use throughout the paper. We let $\MM\{\mathrm{d}F;s\}$ denote the Mellin--Stieltjes transform of a function $F$ of locally bounded variation, that is, $\MM\{\dif F;s\} = \int^{\infty}_{1-} x^{-s} \dif F(x)$, wherever the integral converges. We denote the set of classical prime numbers by $\mathbb{P} = \{2,3,5, \dots \}$ and the set of positive integers by $\N=\{1,2,3,\dotsc\}$. We use the Vinogradov symbol $\ll$ and $O$-notation. Typically, if we write $f(x) \ll g(x)$, the implicit constant only has to be independent of $x$, but is allowed to depend on other parameters appearing in the context. If we do want to emphasize the dependence on an external parameter, we sometimes explicitly mention it in the notation. We write $f(x) = \Omega (g(x))$ if there exist $\eps > 0$ and a sequence $x_n \rightarrow \infty$ such that $|f(x_n)| \geq \eps g(x_n)$.

\section{The logarithmic integrals}
First we recall the definition of the logarithmic integral $\Li$ and that of the related function $\li$: 
\begin{align}\label{eq: Lidef}
	\Li(x) 	&\coloneqq  \int_{1}^{x}\frac{1-u^{-1}}{\log u}\dif u = \sum_{n=1}^{\infty}\frac{1}{n\cdot n!}(\log x)^{n}, \\ \label{eq: lidef}
	\li(x) 		&\coloneqq \sum_{k=1}^{\infty}\frac{\mu(k)}{k}\Li(x^{1/k}) = \sum_{n=1}^{\infty}\frac{1}{n \cdot n!\zeta(n+1)}(\log x)^{n},	
\end{align}
where $\mu$ is the ordinary M\"obius function and $\zeta$ is the Riemann zeta function. We now record some basic properties of these two functions. By M\"obius inversion, we have
\begin{equation}\label{eq:Dirichletinversion}
	\Li(x) = \sum_{k=1}^{\infty}\frac{1}{k}\li(x^{1/k}),
\end{equation}
so $\Li$ relates to $\li$ as $\Pi$ relates to $\pi$.

Note that these functions actually define holomorphic functions on the Riemann surface $\Sigma$ of $\log$. In particular, $\Li(x^z)$ and $\li(x^z)$ have a suitable interpretation for all $z \in \C$. We have
\begin{align}
\label{eq: def Li(x^z)}
	\Li(x^z) & \coloneqq \sum_{n=1}^{\infty}\frac{z^{n}}{n\cdot n!}(\log x)^{n} = \int_{1}^{x}\frac{u^{z-1}-u^{-1}}{\log u}\dif u = \int_{1}^{x^z}\frac{1-v^{-1}}{\log v}\dif v,\\
\label{eq: def li(x^z)}
	\li(x^z) & \coloneqq \sum_{k=1}^{\infty}\frac{\mu(k)}{k}\Li(x^{1/k}) = \sum_{n=1}^\infty \frac{z^n}{n \cdot n! \zeta(n+1)} (\log x)^n,
\end{align}
where the last integral in \eqref{eq: def Li(x^z)} is taken along any path on $\Sigma$ connecting $1$ and $x^z\coloneqq\exp(z\log x)$. Thus,
\begin{equation}\label{eq:lixzprime}
	x \log x \frac{\dif}{\dif x} \li(x^z)= x \log x \sum_{k=1}^\infty \frac{\mu(k)}{k} \od{}{x} \Li(x^{z/k}) = \sum_{k=1}^\infty \frac{\mu(k)}{k} (x^{z/k}-1),
\end{equation}
and the series expansion converges locally uniformly as $|e^w-1|\le e^{|w|}-1\le 2|w|$ for $|w|\le 1$.
Below we will use the last expression in \eqref{eq: def li(x^z)} for $\li(x^z)$, 
as well as the next formula obtained by termwise differentiation:
\begin{equation}\label{eq:liprimeTaylor}
	x \log x \od{}{x} \li(x^z)=\sum_{n=1}^\infty \frac{z^n}{n! \zeta(n+1)} (\log x)^n.
\end{equation}
In particular, for any fixed value $0<\de<1$, we find from this last series expansion the estimate
\begin{equation}\label{eq:liprimexdelta}
	\od{}{x} \li(x^\de) \ge \frac{1}{x \log x} \left( \frac{1}{\zeta(2)} \sum_{n=1}^\infty \frac{\de^n}{n!} (\log x)^n \right) > \frac{ \de}{\zeta(2) x}  > \frac{\de}{2 x}, \qquad x > 1.
\end{equation}

\bigskip

We now calculate the Mellin--Stieltjes transform of $\Li$.
First,
\[
	\MM\{\dif \Li;s\} = \int_1^\infty x^{-s}\dif \Li(x) = \int_1^\infty \frac{x^{-s}-x^{-s-1}}{\log x} \dif x.
\]
It follows that its derivative is $\int_1^\infty (-x^{-s}+x^{-s-1}) \dif x = \frac{1}{s}-\frac{1}{s-1}$ for $\Re s > 1$, hence $\MM\{\dif \Li; s\}=\log s -\log(s-1)+c$ for some $c$. Considering the limit for $s\to \infty$, say along the positive real line, we find $c=0$ and $\MM\{\dif\Li;s\}=\log\frac{s}{s-1}$. Also, $\MM\{\dif\Li(x^z);s\}=\MM\{\dif\Li;s/z\}= \log\frac{s}{s-z}$, which clearly holds for real, positive $z$ by a change of variables, and hence for all $z\in\C$, as for fixed $s$ the map $z\mapsto \MM\{\dif\Li(x^{z});s\}$ is analytic.

\section{The case $1/2\le \beta\le \alpha$}
\label{sec:betalessalpha}

As explained in the introduction, our strategy consists of first constructing template systems in the extended sense, which we then \emph{discretize} by means of Theorem \ref{th: discretization}. The idea is to prescribe suitable poles and zeros for the zeta function. A pole at $s=\om$ of $\zeta_{\MP}(s)$ produces an oscillation of size roughly $x^{\Re \om}$ in both the prime and integer-counting functions, and a zero at $s=\rho$ produces an oscillation of size roughly $x^{\Re \rho}$ in the prime-counting function only. These poles and zeros are introduced by adding the functions $\li(x^{\om})$ and $-\li(x^{\rho})$ to the template prime counting function. First we need a technical lemma which assures that we can always make such a template function increasing by adding a suitably large multiple of $\li(x^{\delta})$ for some small $\delta>0$.

This lemma involves \emph{multisets}. A multiset $\mathcal{A} = (A,m)$ consists of a subset $A$ of $\C$ equipped with a multiplicity function $m:A\to\N$. We say that the multiset is \emph{symmetric} if $\overline{A} = A$ and $m(\overline{a})=m(a)$ for every $a\in A$. In summing over multisets, we sum respective terms according to their multiplicities. Similarly for products. For example, the size of the multiset $\mathcal{A}$ is given by $|\mathcal{A}| = \sum_{a\in\mathcal{A}}1 = \sum_{a\in A}m(a)$, and $\prod_{a\in\mathcal{A}}(s-a) = \prod_{a\in A}(s-a)^{m(a)}$.

\begin{lemma}
\label{lem:increase}
Let $0<\de<1$ and let $\SSS, \RR$ be two finite, symmetric multisets satisfying $\Re \om, \Re \rho \in (0,1)$ for all $\om \in\SSS, \rho\in \RR$. Then there exists a positive number $M_0\coloneqq M_0(\de;\SSS,\RR)$, such that for any $M\ge M_0$ the function $F: [1,\infty) \to \R$, defined as
\begin{equation}\label{eq:Fdef}
F(x)\coloneqq \li(x) +  \sum_{\om \in\SSS}\li(x^{\om}) - \sum_{\rho\in\RR}\li(x^{\rho}) + M\li(x^{\delta}),
\end{equation}
is strictly increasing.
\end{lemma}

\begin{proof}
Recalling the series expansion \eqref{eq: def li(x^z)}
, we note that $F$ is real-valued due to $\SSS$ and $\RR$ being symmetric. Let $Q\ge 1$ be a constant for which $|\om|\le Q$ and $|\rho|\le Q $ for all $\om \in\SSS$ and
$\rho\in \RR$. Further, denote $L \coloneqq |\SSS| +|\RR|$. Applying \eqref{eq:lixzprime} with any $|z|\le Q$, $0<\Re z \le 1$, we obtain
\begin{align*}
	\left| x \log x \dod{}{x} \li(x^z) -(x^z-1) \right| 	
		& \le  2 x^{\Re z/2} \sum_{k=2}^{\lfloor Q \log x\rfloor} \frac{1}{k} + \sum_{k=\lfloor Q \log x\rfloor+1}^\infty \frac{1}{k} \abs[1]{x^{z/k}-1}\\
		& \ll x^{\Re z/2} \log(Q\log x) +  \sum_{k=\lfloor Q \log x\rfloor+1}^\infty \frac{1}{k} \frac{|z|\log x}{k} \ll_Q  \sqrt{x} \log\log x.
\end{align*}
Put $q\coloneqq\max \left(\max_\SSS \Re \om, \max_\RR \Re \rho \right)$. The assumption on the real parts of the elements of the multisets $\SSS$ and $\RR$ entail $q<1$. Writing $F_{0}(x)\coloneqq\li(x) +  \sum_{\om \in\SSS}\li(x^{\om}) - \sum_{\rho\in\RR}\li(x^{\rho})$, the above leads to
\[
	F_{0}'(x)x\log x \ge (x-1) + O_{Q,L}(x^{q} + \sqrt{x} \log \log x).
\]
The right-hand side is eventually positive. So, by taking $M_0$ sufficiently large, we can force that $F'$ is positive on the initial segment too, see \eqref{eq:liprimexdelta}, and therefore also on the whole half-line $[1,\infty)$.
\end{proof}

We now come to the main result of this section which establishes the existence of Beurling systems having a zeta function with prescribed zeros and poles.
It generalizes Theorem 7.4 of \cite{Revesz2022} which dealt only with a multiset of zeros, that is $\SSS=\varnothing$ in the below terminology. This was thus implicitly also a construction of an $[\alpha,\beta]$-system with any given $1/2\le \beta<1$ and some corresponding, unspecified $\alpha \le 1/2$. Here we present the result when both a set of zeros $\RR$ and a set of poles $\SSS$ are prescribed.

\begin{theorem}\label{th: finitepresribtion}
Let $\RR$, $\SSS = (S,m)$ be two finite, symmetric, disjoint multisets and $\Re \om, \Re \rho \in (0,1)$ for all $\om \in\SSS, \rho\in \RR$. Then for each $0<\delta<1/2$ there exists a Beurling number system $(\MP, \MN)$ with the following properties:
\begin{enumerate}
	\item The Chebyshev prime-counting function satisfies
	\[
		\psi_{\MP}(x) = x + \sum_{\om \in \SSS}\frac{x^{\om}}{\om} - \sum_{\rho\in \RR}\frac{x^{\rho}}{\rho} + O(x^{\delta}), \quad x \ge 1.
	\]
	\item The integer-counting function satisfies\footnote{We emphasize that in the sum below, we sum over the underlying set $S$ and not over the multiset $\SSS$. The function $m$ is the multiplicity function of the multiset $\SSS$.}
	\[
		N_{\MP}(x) = ax + \sum_{\substack{\omega\in S\\ \Re \omega > 1/2}}x^{\om}\sum_{j=0}^{m(\omega)-1}b_{\omega, j}(\log x)^{j} + O\bigl\{x^{1/2}\exp\bigl(c(\log x)^{2/3}\bigr)\bigr\}, \quad x \ge 2,
	\]
	for certain constants $a>0$, $b_{\omega, j}\in \C$ and $c>0$, with $b_{\overline{\om}, j}=\overline{b_{\om,j}}$ and $b_{\om,m(\om)-1}\neq0$.
\end{enumerate}
\end{theorem}
We note that the associated zeta function $\zeta_{\MP}(s)$ of the system from the proof is given by $\zeta_{\MP}(s) = E(s)\e^{Z(s)}$, where
\begin{equation}\label{eq:Edef}
	E(s) = E_{M}(s) \coloneqq \frac{s}{s-1} \prod_{\om \in\SSS} \left(\frac{s}{s-\om}\right) \prod_{\rho\in\RR} \left(\frac{{s}-{\rho}}{s}\right) \left(\frac{s}{s-\de}\right)^M,
\end{equation}
for a sufficiently large integer $M$, and $Z(s)$ is a function holomorphic on $\Re s >0$ which satisfies
\begin{equation}\label{eq: zetaestimate}
	Z(\sigma+\I t) \ll \frac{\sigma}{\sigma-1/2} + \sigma\sqrt{\frac{\log (|t|+1)}{\sigma-1/2}}, \quad \text{for }\sigma>1/2.
\end{equation}
In particular, $\zeta_{\MP}(s)$ has meromorphic continuation to the half-plane $\Re s>0$, its zeros in this half-plane are precisely the elements of $\RR$ (with matching multiplicities), and its poles in this half-plane are precisely the elements of $\SSS\cup\{1\}\cup\{\delta\}^{M}$.

\begin{proof}

In view of Lemma \ref{lem:increase}, there is $M$ such that the function $F$ from \eqref{eq:Fdef} is strictly increasing. We now apply the random prime construction result of Theorem \ref{th: discretization} for this function $F$, on noting that $F$ is continuous, obviously satisfies $F(1)=0$, and admits a bound $F(x)\ll x/\log x$. To see this last bound, note that the first integral representation in \eqref{eq: def Li(x^z)} yields $\Li(x^{z}) \ll x^{\Re z}/(\Re z\log x)$. Employing this estimate in the first series representation of \eqref{eq: def li(x^z)} subsequently gives
\begin{align*}
	\li(x^{z}) 	
		&\ll \sum_{k\le|z|\log x}\frac{x^{\Re z/k}}{(\Re z)\log x} + \sum_{k>|z|\log x}\frac{\abs[1]{x^{z/k}-1}}{k}\sup_{v\in [1,x^{z/k}]}\abs[3]{\frac{1-v^{-1}}{\log v}}\\
		&\ll \frac{x^{\Re z}}{(\Re z)\log x} + \sum_{k>|z|\log x}\frac{|z|\log x}{k^{2}} \ll \frac{x^{\Re z}}{(\Re z)\log x}.
\end{align*}
Here we used the second integral representation in \eqref{eq: def Li(x^z)} to bound the terms with $k>|z|\log x$, and the bound $|x^{z/k}-1| \ll (|z|/k)\log x$. The estimate $F(x)\ll x/\log x$ now follows from this and \eqref{eq:Fdef}.

Theorem \ref{th: discretization} thus provides a Beurling prime number system $\mathcal{P}$, $1<p_1<\ldots<p_j<\ldots$, that satifies \eqref{eq: bound exp sums} and whose prime-counting function $\pi_{\MP}$ is such that $|\pip-F|\le 2$. We let $\MN$ be the system of integers generated by the primes $\MP$.

By definition, $\Pip(x)\coloneqq\sum_{k=1}^\infty \frac{1}{k} \pip(x^{1/k}) = \sum_{k\le \log x/\log p_1} \frac1{k} \pip(x^{1/k})$, the terms becoming identically zero as soon as $k>\log x/\log p_1$.
We thus find
\begin{align*}
	\Pip(x)	&=\sum_{k\le \log x/\log p_1} \frac1{k} (F(x^{1/k}) +O(1)) =\sum_{k\le \log x/\log p_1} \frac1{k} F(x^{1/k}) + O(\log (\log x/\log p_1))\nonumber\\
			&=\sum_{k=1}^{\infty}\frac{1}{k}F(x^{1/k}) +O\biggl(\sup_{y\in [1,p_{1}]}F'(y)\sum_{k>\frac{\log x}{\log p_{1}}}\frac{x^{1/k}-1}{k} + \log\log x\biggr)\nonumber,
\end{align*}
where we used the mean value theorem to bound $F(x^{1/k}) = F(x^{1/k})-F(1)$ for the remaining part of the series with $x^{1/k} < p_{1}$. As $x^{1/k}-1 \ll \log (x)/k$ for $x^{1/k} < p_{1}$, we obtain
\begin{equation}
\label{eq:Pipx}
	\Pip(x) = G(x)+ O(\log\log x),
\end{equation}
where we have set
\begin{equation} \label{eq:Gdef}
	G(x) \coloneqq \sum_{k=1}^{\infty}\frac{1}{k}F(x^{1/k}) = \Li(x) +  \sum_{\om \in\SSS}\Li(x^{\om}) - \sum_{\rho\in\RR}\Li(x^{\rho}) + M\Li(x^{\delta}),
\end{equation}
and where the latter equality holds in view of \eqref{eq:Dirichletinversion} and \eqref{eq:Fdef}.

Since $\psip(x)\coloneqq\int_1^x \log u \dif\Pip(u)$, here we compute
\[
	\int_1^x \log u \dif \Li(u^z) = \int_1^x \log u \frac{u^{z-1}-u^{-1}}{\log u} \dif u = \frac{x^z - 1}{z}-\log x,
\]
resulting in
\[	
	\psip(x)= x + \sum_{\om \in \SSS}\frac{x^{\om}}{\om} - \sum_{\rho\in \RR}\frac{x^{\rho}}{\rho} + M \frac{x^\de}{\de} + O(\log x \log\log x),
\]
which in particular demonstrates the first assertion.

Next we set to the computation of
\begin{align}\label{eq:logzeta} \notag
\log\zs&=\sum_{p_{j}\in \MP} \log \Bigl(\frac{1}{1-p_{j}^{-s}}\Bigr)=\int_{1}^\infty x^{-s} \dif\Pip(x)=\MM\{\dif\Pip;s\}.
\end{align}
Formula \eqref{eq:Pipx} leads to the approximate evaluation
\[
	\log\zs \approx \MM\{\dif G;s\}=\log\frac{s}{s-1} + \sum_{\om \in\SSS} \log\frac{s}{s-\om} + \sum_{\rho \in\RR} \log\frac{s-\rho}{s} + M\log\frac{s}{s-\de}.
\]
Moreover, the error $Z(s) \coloneqq \MM\{\dif \Pi_{\MP}- \dif G;s\}$ here arises as the Mellin--Stieltjes transform of the error function $\Pip(x)-G(x)$, which is $O(\log\log x)$. Hence $Z(s)$ is analytic for $\Re s>0$.
Recalling the definition \eqref{eq:Edef} of $E(s)$, we have $E(s)=\exp(\MM\{\dif G;s\})$, so that $\zeta_{\MP}(s) = E(s)\e^{Z(s)}$. This already yields the desired analytic behavior of $\zeta_{\MP}(s)$. However, to deduce the asymptotic behavior of the integer-counting function $N_{\MP}(x)$ we need the bound \eqref{eq: zetaestimate} for $Z(s)$.
The modulus of $Z(s)$ is not well-controlled automatically. More precisely, here we need a recourse to the full strength of the error estimates furnished by Theorem \ref{th: discretization} for $\pip(x)$, which we formulate in terms of the auxiliary function
\[
	J(x,t)\coloneqq\sum_{p_j\le x} p_j^{-\I t} - \int_1^x u^{-\I t} \dif F(u)=\int_1^x u^{-\I t} \dif\,(\pip(u)-F(u)).
\]
With this notation, Theorem \ref{th: discretization} gives
\begin{equation}\label{eq:Jfromthms}
	|J(x,t)| \ll \sqrt{x} + \sqrt{x\frac{\log(|t|+2)}{\log(x+2)}}.
\end{equation}
Let us write $Z(s)$ as
\[
	Z(s) = \MM\{\dif\Pi_{\MP}-\dif G;s\} = \MM\{\dif\Pi_{\MP}-\dif \pi_{\MP};s\} + \MM\{\dif\pi_{\MP}-\dif F;s\} - \MM\{\dif G-\dif F;s\}.
\]
The second term here is
\[
	R(s) \coloneqq \MM\{\dif \pip-\dif F;s\} = \int_1^\infty x^{-s} \dif\,(\pip(x)-F(x))=\int_1^\infty x^{-\si} \dif J(x,t),
\]
where the last integral is with respect to the dummy variable $x$, while $t = \Im s$ is fixed. From here partial integration leads to
\[
	R(s)=\left[x^{-\si} J(x,t)\right]_1^\infty +\si \int_1^\infty x^{-\si-1} J(x,t) \dif x=\si \int_1^\infty x^{-\si-1} J(x,t) \dif x
\]
for $\sigma > 1/2$, so that \eqref{eq:Jfromthms} furnishes
\begin{align}
	|R(\sigma+\I t)|
		&\ll \si\int_{1}^{\infty}x^{-\si-1/2}\dif x + \si\sqrt{\log(|t|+2)}\int_{1}^{\infty}x^{-\si-1/2}(\log x)^{-1/2}\dif x \nonumber\\
		&\ll \frac{\si}{\si-1/2} + \frac{\si}{\sqrt{\si-1/2}}\sqrt{\log(|t|+2)}, \quad \si>1/2, \label{eq:Rfirstesti}
\end{align}
upon a change of variables $y=(\si-1/2)\log x$ in the second integral.

For the other two terms we exploit the monotonicity of $\Pip(x)-\pip(x)$ and $G(x)-F(x)$, and the standard estimates $\Pip(x)-\pip(x)\ll \sqrt{x}$, $G(x)-F(x)\ll \sqrt{x}$. This monotonicity may be deduced from the termwise monotonicity of the summands in $\Pip(x)-\pip(x) = \sum_{k=2}^{\infty}\pip(x^{1/k})/k$ and $G(x)-F(x) = \sum_{k=2}^{\infty}F(x^{1/k})/k$. As for the claimed estimates, this follows from $F(x) \ll x/\log x$ as $x\to \infty$ and $F(x) \ll x-1$ as $x\to 1$, so that
\[
	G(x) - F(x) = \sum_{k=2}^{\infty}\frac{F(x^{1/k})}{k} \ll \sum_{2\le k\le \log x}\frac{x^{1/k}}{\log x} + \sum_{k>\log x}\frac{x^{1/k}-1}{k} \ll x^{1/2}.
\]
The estimate for $\Pip(x)-\pip(x)$ can be deduced similarly, the summands in the sum now being zero if $k> \log x/\log p_{1}$.
Using all four of these properties, we obtain
\begin{align*}
	\abs{\MM\{\dif\Pip-\dif\pip;s\}} 	&\le \int_1^\infty x^{-\si} \dif\,(\Pip(x)-\pip(x))\ll \frac{\si}{\si-1/2},\\
	\abs{\MM\{\dif G-\dif F;s\}}		&\le \int_{1}^{\infty}x^{-\si}\dif\,(G(x)-F(x)) \ll \frac{\si}{\si-1/2},
\end{align*}
for $\si>1/2$.
Collecting the above estimates we arrive at \eqref{eq: zetaestimate}, which is, as mentioned above, an essential element for the proof of the theorem's second part.

\medskip

To establish the asymptotic behavior of $N_{\mathcal{P}}(x)$, we perform a Perron inversion in the spirit of \cite[Th.\ 17.11]{DiamondZhangbook}. Since $N_{\mathcal{P}}(x)$ is non-decreasing, we have
\begin{align}
	N_{\mathcal{P}}(x) 	&\le \int_{x}^{x+1}N_{\mathcal{P}}(u)\dif u = \frac{1}{2\pi\I}\int_{2-\I\infty}^{2+\I\infty}\frac{(x+1)^{s+1}-x^{s+1}}{s(s+1)}\zeta_{\MP}(s)\dif s, \label{eq: upper Perron}\\
	N_{\mathcal{P}}(x)	&\ge \mathrlap{\int_{x-1}^{x}N_{\mathcal{P}}(u)\dif u}\phantom{\int_{x}^{x+1}\NP(u)\dif u}
						 = \frac{1}{2\pi\I}\int_{2-\I\infty}^{2+\I\infty}\frac{x^{s+1}-(x-1)^{s+1}}{s(s+1)}\zeta_{\MP}(s)\dif s. \label{eq: lower Perron}
\end{align}
Note that the contour integrals converge absolutely. We limit ourselves to the analysis of the first Perron integral \eqref{eq: upper Perron}, the analysis of the second being similar.

Let $\sigma_{0}$ be a real number in $(\delta,1/2)$ distinct from $\Re \omega$ for each $\om\in S$, let $T_{0}$ be a real number satisfying $T_{0}> \Im \om$ for each $\om\in S$, and let $\sigma_{x}=1/2+(\log x)^{-1/3}$. We transfer the integration contour to the broken line
\begin{align*}
	\Gamma 	&\coloneqq 	(\sigma_{x}-\I\infty, \sigma_{x}-\I T_{0}] \cup [\sigma_{x}-\I T_{0}, \sigma_{0}-\I T_{0}] \cup [\sigma_{0}-\I T_{0}, \sigma_{0}+ \I T_{0}] \\
			&\phantom{\coloneqq {}} \cup [\sigma_{0}+\I T_{0}, \sigma_{x}+\I T_{0}] \cup [\sigma_{x}+\I T_{0}, \sigma_{x}+\I \infty),
\end{align*}
which is justified in view of the bound \eqref{eq: zetaestimate}.
By the residue theorem we get
\begin{align}
	\int_{x}^{x+1}N_{\MP}(u) \dif u
		&=a(x+1/2) + \Biggl[\sum_{\substack{\om\in S\\ \Re \om > \sigma_{0}}}u^{\om+1}\sum_{j=0}^{m(\om)-1}\tilde{b}_{\om, j}(\log u)^{j}\Biggr]^{x+1}_{x} \label{eq: main terms}\\
		&\phantom{= {}} + \frac{1}{2\pi\I}\int_{\Gamma}\frac{(x+1)^{s+1}-x^{s+1}}{s(s+1)}\zeta_{\MP}(s)\dif s, \label{eq: remainder int}
\end{align}
where $a>0$ is the residue of $\zeta_{\MP}(s)$ at $s=1$, and for certain constants $\tilde{b}_{\om,j}$ with $\tilde{b}_{\overline{\om},j} = \overline{\tilde{b}_{\om,j}}$ and $\tilde{b}_{\om, m(\om)-1}\neq0$. Observe that
\begin{align*}	
	\bigl[u^{\om+1}(\log u)^{j}\bigr]^{x+1}_{x} 	&= (\om+1)x^{\om}(\log x)^{j} + jx^{\om}(\log x)^{j-1} + O\bigl(x^{\Re\om-1}(\log x)^{j}\bigr), \quad \text{and also}\\
	\bigl[u^{\om+1}(\log u)^{j}\bigr]^{x}_{x-1} 	&= (\om+1)x^{\om}(\log x)^{j} + jx^{\om}(\log x)^{j-1} + O\bigl(x^{\Re\om-1}(\log x)^{j}\bigr).
\end{align*}	
Hence the main terms in \eqref{eq: main terms} are of the form
\[
	ax + \sum_{\substack{\omega\in S\\ \Re \omega > 1/2}}x^{\om}\sum_{j=0}^{m(\omega)-1}b_{\omega, j}(\log x)^{j} + O\bigl(x^{1/2}(\log x)^{|\SSS|}\bigr)
\]
for certain constants $b_{\omega,j}$ and it remains to show that the integral \eqref{eq: remainder int} admits the error term $O\bigl\{x^{1/2}\exp\bigl(c(\log x)^{2/3}\bigr)\bigr\}$ for some $c>0$.

To do so we shall bound $(x+1)^{s+1}-x^{s+1}$ by $\abs{s}x^{\sigma}$ for $\abs{t}\le x$ and by $O(x^{\sigma+1})$ for $\abs{t}\ge x$. The part of the integral along the three segments with $\abs{t} \le T_{0}$ is clearly $\ll x^{\si_{x}}=x^{1/2}\exp\bigl((\log x)^{2/3}\bigr)$. In view of \eqref{eq: zetaestimate} there is a constant $C>0$ such that
\[
	\abs{Z(\sigma_{x}+\I t)} \le C\bigl((\log x)^{1/3} + (\log x)^{1/6}\sqrt{\log(\,\abs{t}+1)}\bigr),
\]
so that for $\abs{t}\ge T_{0}$,
\[
	|\zeta_{\MP}(\sigma_{x} + \I t)|  \ll \exp\bigl\{C\bigl((\log x)^{1/3} + (\log x)^{1/6}\sqrt{\log(\,\abs{t}+1)}\bigr)\bigr\}.
\]
Now
\begin{align*}
	\int_{T_{0}}^{x}\exp\bigl(C\bigl((\log x)^{1/3} + (\log x)^{1/6}\sqrt{\log t}\bigr)\frac{\dif t}{t} 		&\ll \exp\bigl(2C(\log x)^{2/3}\bigr)(\log x),\quad \text{and} \\
	\int_{x}^{\infty}\exp\bigl(C\bigl((\log x)^{1/3} + (\log x)^{1/6}\sqrt{\log t}\bigr)\frac{\dif t}{t^{2}} 	&\ll \frac{1}{x} \exp\bigl(2C(\log x)^{2/3}\bigr),
\end{align*}
so that indeed
\[
	\frac{1}{2\pi\I}\int_{\Gamma}\frac{(x+1)^{s+1}-x^{s+1}}{s(s+1)}\zeta_{\MP}(s)\dif s \ll x^{1/2}\exp\bigl(c(\log x)^{2/3}\bigr)
\]
for any $c>2C+1$. This finishes the proof.	
\end{proof}

\begin{corollary}\label{cor:alphagreaterbeta} Let $1/2 \le \beta \le \alpha < 1$.
Then there exists an $[\alpha,\beta]$-Beurling system $(\MP,\MN)$.
\end{corollary}
\begin{proof}
This immediately follows from the theorem upon selecting $\RR = \{\alpha\}$, $\SSS=\{\beta\}$ if $\alpha > \beta$, and $\RR=\varnothing$, $\SSS=\{\beta\}$ if $\alpha=\beta$. In the case $\beta=1/2$, the second assertion of the theorem only yields $N_{\MP}(x)  = ax + O\bigl\{x^{1/2}\exp\bigl(c(\log x)^{2/3}\bigr)\bigr\}$, but no oscillation result. However, with the presence of a pole of $\zeta_{\MP}(s)$ at $s=\beta=1/2$, it is clear from $\zeta_{\MP}(s)-as/(s-1)=s\int_1^{\infty}x^{-(s+1)}(N_{\MP}(x)-ax)\dif x$ that $N_{\MP}(x) - ax \ll x^{1/2-\eps}$ cannot hold for any $\eps>0$, as that would make the integral absolutely convergent and the function $\zeta_{\MP}$ analytic around 1/2.
\end{proof}
\begin{corollary}
Let $0\le\alpha<1/2$. Then there exists an $[\alpha, 1/2]$-Beurling system.
\end{corollary}
\begin{proof}
This follows from Theorem \ref{th: finitepresribtion} with $\RR=\{\alpha\}$, $\SSS=\varnothing$ if $\alpha>0$. For $\alpha=0$, we may take $M = 0$ and $\RR=\SSS=\varnothing$ since $\li(x)$ is non-decreasing. Then the example in (the proof of) Theorem \ref{th: finitepresribtion} is independent of $\delta$ which guarantees that the primes of this system are indeed $0$-well-behaved. 

The fact that $N_{\MP}(x)-ax\ll x^{1/2-\eps}$ cannot hold for any $\eps>0$ follows from Hilberdink's result $\max\{\alpha,\beta\}\ge1/2$.
\end{proof}

\begin{remark} \:
\begin{enumerate}
\item We already mentioned that in \cite[Ch.\ 17]{DiamondZhangbook} a $[1,\beta_0]$-system $\mathcal{P}$ was constructed for some $\beta_0 \leq 1/2$. One may generate from this example $[1,\beta]$-systems for arbitrary $\beta_0 < \beta < 1$ as follows. Let $\mathcal{P}_{\beta} = \mathcal{P} \cup \{p^{1/\beta}| p \in \mathbb{P}\}$. The primes of $\mathcal{P}_{\beta}$ remain not well-behaved as $\pi_{\mathcal{P}_{\beta}}(x) = \pi_{\MP}(x) + O(x^{\beta})$, while the extra primes introduce a pole for $\zeta_{\mathcal{P}_{\beta}}(s) = \zeta_{\MP}(s) \zeta(s/\beta)$ at $s = \beta$. This guarantees that the integers are not better than $\beta$-well-behaved as otherwise $\zeta_{\mathcal{P}_{\beta}}$ would have been analytic at $s = \beta$ due to a similar argument as in Corollary \ref{cor:alphagreaterbeta}. The $\beta$-well-behavedness itself follows from
 \begin{align*}
 N_{\MP_{\beta}}(x) & = \sum_{\substack{n^{1/\beta}\le x\\n\in\N}}N_{\MP}(xn^{-1/\beta}) \\
 & = \sum_{n\in\N} \frac{ax}{n^{1/\beta}} +  x \sum_{\substack{n^{1/\beta}> x\\n\in\N}}O(n^{-1/\beta}) + x^{\beta_0 + \eps}  \sum_{\substack{n^{1/\beta} \le x\\n\in\N}}O(n^{-(\beta_0 + \eps)/\beta})\\
 & = ax\zeta(1/\beta) + O(x^{\beta}).
 \end{align*}

\item If we were only interested in establishing Corollary \ref{cor:alphagreaterbeta} it is also an option to apply the discretization procedure directly to the function $G$ from \eqref{eq:Gdef} where the function $\li$ is replaced with the simpler function $\Li$ and where we may even take $M = 0$. The reason this does not generate problems is that the error introduced in going from $\li$ to $\Li$ is $O(x^{1/2})$ which is negligible in case $\alpha \geq 1/2$.
\end{enumerate}
\end{remark}

\section{The case $\alpha < \beta$}
\label{sec:alphalessbeta}

Assume that $\alpha, \beta \in [0,1)$ with $\alpha < \beta$ and $\beta > 1/2$. As in the previous section, we first construct a system in the extended sense by prescribing certain properties of its zeta function. Again, a zero with real part equal to $\alpha$ will be used to create the desired behavior of the primes. Unlike the previous case, it is not possible to introduce a pole with real part equal to $\beta$ on $\zeta(s)$ to create the desired behavior of the integers. Indeed, as $\beta>\alpha$, this pole will destroy the $\alpha$-well-behavedness of the primes as it will also generate a singularity for $\log \zeta(s)$. Instead, we construct a zeta function exhibiting extreme growth to the left of the line $\Re s = \beta$. For this, we take inspiration from the construction of \cite{BDV2020} (see also \cite{BDV2022}).

Let $(\tau_{l})_{l=1}^{\infty}$ be a sequence which rapidly increases to $\infty$, $(\delta_{l})_{l=1}^{\infty}$ a sequence tending to $0$, and $(\nu_{l})_{l=1}^{\infty}$ a sequence in the interval $[2, 3]$. We set $A_{l} = \tau_{l}^{1+\delta_{l}}$ and $B_{l} = \tau_{l}^{\nu_{l}}$. Later we will specify further how to choose the sequences $(\tau_{l})_{l}$, $(\delta_{l})_{l}$, and $(\nu_{l})_{l}$, but from the outset we assume that $(\tau_{l})_{l}$ increases sufficiently rapidly so that $B_{l} < A_{l+1}$. We consider the (absolutely continuous) function $\Pi_{C}$ defined as
\[
	\Pi_{C}(x) = \Li(x) - \Li(x^{\alpha}) + \sum_{l=1}^{\infty}R_{l}(x),
\]
where $R_{l}(x) = \int_{1}^{x}\dif R_{l}(u)$ with
\[
	\dif R_{l}(u) = \begin{cases}
		\tau_{l}\cos(\tau_{l}\log u)u^{\beta-2}\dif u 		&\text{if } A_{l}\le u < B_{l};\\
		0									&\text{else}.
	\end{cases}
\]
For technical convenience we assume that $\tau_{l} \log A_{l}, \tau_{l} \log B_{l}\in 2\pi\Z$, by changing the values of $\delta_{l}$ and $\nu_{l}$ by a $o(1)$ amount if necessary. A small calculation yields
\[
	R_{l}(x) = \frac{\tau_{l}}{(1-\beta)^{2}+\tau_{l}^{2}}\Bigl\{(1-\beta)\bigl(A_{l}^{\beta-1} - x^{\beta-1}\cos(\tau_{l}\log x)\bigr) + \tau_{l}x^{\beta-1}\sin(\tau_{l}\log x)\Bigr\}
\]
for $A_{l}\le x \le B_{l}$, and $R_{l}(x) = R_{l}(B_{l})$ for $x > B_{l}$. For $A_{l}\le x < B_{l}$ we have $R_{l}(x) \ll x^{\beta-1}$, and assuming the $\tau_{l}$ increase sufficiently rapidly, $\sum_{l}R_{l}(B_{l}) < \infty$. Hence we have
\begin{equation}
\label{eq: asymp Pi_{C}}
	\Pi_{C}(x) = \Li(x) - \Li(x^{\alpha}) + O(1).
\end{equation}
The function $\Pi_{C}(x)$ is non-decreasing because for $A_{l} < x < B_{l}$,
\[
	\Pi_{C}'(x) = \frac{1-x^{\alpha-1}}{\log x} + \tau_{l}\cos(\tau_{l}\log x)x^{\beta-2} \ge \frac{1-x^{\alpha-1}}{\log x} - x^{\beta-1} \ge 0,
\]
provided that $A_{1}$ is sufficiently large.

The associated zeta function $\zeta_{C}(s)$ is given by
\[
	\zeta_{C}(s) = \frac{s-\alpha}{s-1}\exp\biggl(\sum_{l=1}^{\infty}\eta_{l}(s)\biggr),
\]
where
\begin{align*}
	\eta_{l}(s) \coloneqq \int_{1}^{\infty}x^{-s}\dif R_{l}(x) & = \frac{\tau_l}{2}\int^{B_l}_{A_l} \bigl(x^{\beta - 2 -s + \I \tau_{l}} + x^{\beta - 2 -s - \I \tau_{l}}\bigr) \dif x  \\
	& = \frac{\tau_{l}}{2}\bigl(A_{l}^{\beta-1-s}-B_{l}^{\beta-1-s}\bigr)\biggl(\frac{1}{s+1-\beta-\I\tau_{l}}+\frac{1}{s+1-\beta+\I\tau_{l}}\biggr).
\end{align*}
Due to the rapid increase of the sequence $\tau_{l}$, we have for each $t\ge0$ that $\abs{t-\tau_{l}} \gg \tau_{l}$ except for at most one value of $l$. Hence the series $\sum_{l}\eta_{l}(s)$ converges locally uniformly in the half-plane $\Re s > \beta-1$.
Let $s=\sigma+\I t$, with $t\ge0$, and let $l_{0}$ be such that $l\neq l_{0} \implies \abs{t-\tau_{l}} \gg \tau_{l}$. If $\sigma\ge \beta$, then
\begin{equation}
\label{eq: bound eta for sigma ge beta}
	\sum_{l=1}^{\infty}\eta_{l}(s) \ll \sum_{l\neq l_{0}}A_{l}^{-1} + \tau_{l_{0}}A_{l_{0}}^{-1} \ll 1 + \tau_{l_{0}}^{-\delta_{l_{0}}} \ll 1.
\end{equation}
On the other hand, if $0<\sigma< \beta$, then as the sequence $A_l$ grows sufficiently fast we have
\[
	\sum_{l=1}^{\infty}\eta_{l}(\sigma+\I\tau_{l_{0}}) = \eta_{l_{0}}(\sigma+\I\tau_{l_{0}}) + O\biggl(\sum_{l\neq l_{0}} A_l^{-(1-\beta)}\biggr) = \eta_{l_{0}}(\sigma+\I\tau_{l_{0}}) + O(1),
\]
and
\begin{align*}
	 \eta_{l_{0}}(\sigma+\I\tau_{l_{0}}) 	&= \frac{\tau_{l_{0}}A_{l_{0}}^{\beta-1-\sigma}}{2(1+\sigma-\beta)} + O\bigl(\tau_{l_{0}}B_{l_{0}}^{\beta-1-\sigma} + 1\bigr) \\
	 							&= \frac{\tau_{l_{0}}^{(1+\delta_{l_{0}})(\beta-\sigma) - \delta_{l_{0}}} }{2(1+\sigma-\beta)}+ O\bigl(\tau_{l_{0}}^{2(\beta-\sigma)-1}+1\bigr),
\end{align*}
in view of $A_{l} = \tau_{l}^{1+\delta_{l}}$ and $B_{l} = \tau_{l}^{\nu_{l}}\ge \tau_{l}^{2}$. For any $\eps > 0$  , we may take $l_{0}$ sufficiently large such that $\delta_{l_{0}}\le \max\{\eps, 1-\beta\}$, whence for $0 < \sigma < \beta$,
\begin{equation}
\label{eq: extreme growth zeta_{C}}
	\zeta_{C}(\sigma+\I\tau_{l_{0}}) \gg \exp\biggl\{\frac{\tau_{l_{0}}^{\beta-\sigma-\eps}}{2(1+\sigma-\beta)}\bigl(1+o(1)\bigr)\biggr\}.
\end{equation}
In particular, $\zeta_{C}$ is not polynomially bounded in any half plane $\sigma \ge \sigma_{0}$ if $\sigma_{0}< \beta$.

Next we consider the function $\pi_{C}(x)$, which is defined through
\[
	\Pi_{C}(x) = \sum_{k=1}^{\infty}\frac{\pi_{C}(x^{1/k})}{k}, \quad \text{i.e.} \quad \pi_{C}(x) = \sum_{k=1}^{\infty}\frac{\mu(k)}{k}\Pi_{C}(x^{1/k}),
\]
$\mu(k)$ being the classical M\"obius function. Using \eqref{eq: lidef}, we get
\[
	\pi_{C}(x) = \li(x) - \li(x^{\alpha}) + \sum_{k=1}^{\infty}\sum_{l=1}^{\infty}r_{l,k}(x),
\]
where $r_{l,k}(x) = (\mu(k)/k)R_{l}(x^{1/k})$. For each fixed $x\ge1$, only finitely many terms in the sum are nonzero.

We now prove that $\pi_{C}(x)$ is non-decreasing by showing that its derivative is $\ge 0$ (almost everywhere). The argument is similar to the one given in \cite[Section 4]{BrouckeVindas}. We have
\[
	r_{l,k}'(x) = \begin{dcases}
		\frac{\mu(k)}{k^{2}}\tau_{l}\cos\Bigl(\frac{\tau_{l}}{k}\log x\Bigr)x^{-\frac{1-\beta}{k}-1}	&\text{if } A_{l}^{k} < x < B_{l}^{k}, \\
		0																&\text{if } x < A_{l}^{k} \text{ or } x > B_{l}^{k}.
	\end{dcases}
\]
For $x=A_{l}^{k}, B_{l}^{k}$, we leave the derivative undefined.
We set $I_{l,k} = (A_{l}^{k}, B_{l}^{k})$. Let now $x \ge 1$ be fixed. We observe that for every $k$ there is at most one value of $l$ such that $x\in I_{l,k}$. Also, if $k_{1}\le k_{2}$ and $l_{1}$, $l_{2}$ are such that $x\in I_{l_{1},k_{1}}\cap I_{l_{2},k_{2}}$, then $l_{1} \ge l_{2}$. Let now $m$ be the smallest integer such that $x\in I_{l,m}$ for some $l$, and let $L$ be the unique index such that $x\in I_{L,m}$. By the aforementioned observations, and using that $x\in I_{l,k}$ implies $x\ge \tau_{l}^{k}$, we get
\[
	\abs[3]{\biggl(\sum_{k=1}^{\infty}\sum_{l=1}^{\infty}r_{l,k}(x)\biggr)'} \leq \sum_{\substack{k,l\\ x\in I_{l,k}}}\abs[1]{r_{l,k}'(x)} \le \frac{1}{x}\sum_{k=m}^{\infty}\frac{\tau_{L}^{\beta}}{k^{2}} \le 2x^{\beta-1}.
\]

If $x < A_{1}$, then $\pi_{C}'(x) = \li'(x) - \bigl(\li(x^{\alpha})\bigr)' \ge 0$, because in view of \eqref{eq:liprimeTaylor} we find
\[
	 \bigl(\li(x^{\alpha})\bigr)' = \sum_{m=1}^{\infty}\frac{\alpha^{m}}{m!\zeta(m+1)}\frac{(\log x)^{m-1}}{x} \le \sum_{m=1}^{\infty}\frac{1}{m!\zeta(m+1)}\frac{(\log x)^{m-1}}{x} = \li'(x).
\]
Since
\[
	\frac{1}{\zeta(2)}\sum_{m=1}^{\infty}\frac{(\log x)^{m-1}}{m! x} \le \li'(x) \le \sum_{m=1}^{\infty}\frac{(\log x)^{m-1}}{m! x} = \frac{1-x^{-1}}{\log x},
\]
we have for $x > A_{1}$ that
\[
	\pi_{C}'(x) \ge \frac{1/\zeta(2) - x^{\alpha-1}}{\log x} - 2x^{\beta-1} > 0,
\]
provided that $A_{1}$ is sufficiently large. Hence $\pi_{C}(x)$ is non-decreasing. We clearly also have $\pi_{C}(x) \ll x/\log x$.

We now apply\footnote{If $\alpha \ge 1/2$, it is also possible to apply the theorem with $F(x) = \Pi_{C}(x)$, as $\Pi_{\MP}(x) = \pi_{\MP}(x) + O(x^{1/2}) = \Pi_{C}(x) + O(x^{1/2})$.} Theorem \ref{th: discretization} with $F(x) = \pi_{C}(x)$. This yields a sequence of generalized primes $\MP = (p_{j})_{j=1}^{\infty}$ satisfying $\pi_{\MP}(x) = \pi_{C}(x) + O(1)$ and
\[
	\abs[3]{\sum_{p_{j}\le x}\frac{1}{p_{j}^{\I t}} - \int_{1}^{x}u^{-\I t}\dif\pi_{C}(u)} \ll \sqrt{x} + \sqrt{\frac{x\log(\,\abs{t}+1)}{\log(x+1)}}.
\]
Let $\MN$ be the sequence of generalized integers generated by $\MP$. Then $(\MP,\MN)$ is an $[\alpha,\beta]$-system. Indeed, as before
\[
	\Pi_{\MP}(x) = \sum_{k\le\log x/\log p_{1}}\frac{\pi_{\MP}(x^{1/k})}{k} = \sum_{k\le\log x/\log p_{1}}\frac{\pi_{C}(x^{1/k})}{k} + O(\log\log x)
	= \Pi_{C}(x) + O(\log\log x).
\]
Hence, in view of \eqref{eq: asymp Pi_{C}}, the Chebyshev prime-counting function $\psi_{\MP}(x)$ satisfies
\[ \psi_{\MP}(x) = \int^{x}_{1} \log u \Big(\frac{1-u^{-1}}{\log u} \dif u - \frac{u^{\alpha - 1}-u^{-1}}{\log u} \dif u + \dif O(1) \Big) = x - \frac{x^{\alpha}}{\alpha} + O(\log x),
\]
and thus the primes are $\alpha$-well-behaved. 
To deduce the behavior of $N_{\MP}(x)$, we pass through the zeta function.

We have
\[
	\zeta_{\MP}(s) = \zeta_{C}(s)\e^{Z(s)},
\]
where
\[
	Z(s) \coloneqq \int_{1}^{\infty}x^{-s}\dif \,\bigl(\Pi_{\MP}(x)-\Pi_{C}(x)\bigr) \\
\]
As in the proof of Theorem \ref{th: finitepresribtion}, we obtain the bound \eqref{eq: zetaestimate} for $Z(s)$. In fact we will only need $Z(\sigma+\I t)\ll_{\si_{0}} \sqrt{\log(\,\abs{t}+2)}$ on every fixed closed half-plane $\sigma \ge \sigma_{0}$ with $\sigma_{0}> 1/2$.

To establish the asymptotic behavior of $N_{\MP}(x)$, we perform a Perron inversion. The analysis is very similar to the Perron inversion in the proof of Theorem \ref{th: finitepresribtion}.

Since $N_{\MP}(x)$ is non-decreasing, we have
\begin{align} \label{eq: secondPerron}
	N_{\MP}(x) 	&\le \int_{x}^{x+1}N_{\MP}(u)\dif u = \frac{1}{2\pi\I}\int_{2-\I\infty}^{2+\I\infty}\frac{(x+1)^{s+1}-x^{s+1}}{s(s+1)}\zeta_{\MP}(s)\dif s, \\
	N_{\MP}(x)	&\ge \mathrlap{\int_{x-1}^{x}N_{\MP}(u)\dif u}\phantom{\int_{x}^{x+1}N_{\MP}(u)\dif u}
						 = \frac{1}{2\pi\I}\int_{2-\I\infty}^{2+\I\infty}\frac{x^{s+1}-(x-1)^{s+1}}{s(s+1)}\zeta_{\MP}(s)\dif s,
\end{align}
and where the contour integrals converge absolutely. We shift the contour of integration to the line $\si=\beta$, picking up the residue $a(x+O(1))$ (with $a\coloneqq\res_{s=1}\zeta_{\MP}(s)$) from the pole at $s=1$. Let $C>0$ be such that $\abs[0]{Z(s)} \le C\sqrt{\log(\,\abs{t}+2)}$ for $\sigma\ge\beta$. Then $\zeta_{\MP}(\sigma+\I t) \ll \exp\bigl(C\sqrt{\log(\,\abs{t}+2)}\bigr)$ for $\si\ge\beta$. Considering the ranges $\abs{t}\le x$ and $\abs{t}>x$ separately as before, we arrive at
\[
	N_{\MP}(x) \le ax + O\bigl\{x^{\beta}\exp\bigl(C\sqrt{\log x}\bigr)\log x\bigr\} \quad \text{and} \quad N_{\MP}(x) \ge ax + O\bigl\{x^{\beta}\exp\bigl(C\sqrt{\log x}\bigr)\log x\bigr\}.
\]
In particular $N_{\MP}(x) = ax+O_{\eps}(x^{\beta+\eps})$ for every $\eps>0$.

If now $N_{\MP}(x) = ax + O(x^{\beta-\eps})$ for some $\eps>0$, then $\zeta_{\MP}(s)=as/(s-1)+s\int_1^{\infty}x^{-(s+1)}(N_{\MP}(x)-ax)\dif x$ in the half-plane $\sigma \ge \sigma_{0}\coloneqq \beta-\eps/2$ say, and so $\zeta_{\MP}(s) -a/(s-1)$ is polynomially bounded there. 
In view of $\zeta_{\MP}(s) = \zeta_{C}(s)\e^{Z(s)}$ and $Z(s) \ll_{\sigma_{0}} \sqrt{\log(\,\abs{t}+2)}$, this contradicts \eqref{eq: extreme growth zeta_{C}}. This finishes the proof that $(\MP,\MN)$ is an $[\alpha, \beta]$-system.

\begin{remark} The example that we presented here is inspired on that of \cite{BDV2022} where a $[0,1]$-system $\mathcal{P}$ was constructed. Our construction above does not fully go through to get $[\alpha,1]$-systems for arbitrary $0 < \alpha < 1$, but to get those one can argue as follows. We first remark that any Beurling system $(\mathcal{P}, \mathcal{N})$ for which $\zeta_{\MP}(s) - a/(s-1)$ admits polynomial bounds in a half-plane $\sigma \geq \sigma_{0}$, $\sigma_0 < 1$, must have well-behaved integers. Indeed, if $\zeta(\sigma+\I t)\ll |t|^{M}$, say, in the half-plane $\sigma\ge \sigma_0$ away from the pole at $s = 1$, then we may consider the non-decreasing functions $N_{m}(x) := \int^{x}_{1}  u^{-1}N_{m-1}(u) \dif u$ and $N_{0}(x) := N_{\MP}(x)$, for which $\mathcal{M}\{\dif N_m; s\} =  \zeta_{\MP}(s)/s^m$, so that for $m>M$ the exact same Perron inversion argument as performed in the examples above would provide the asymptotic formula $N_{m}(x) = ax + O_{\eps}(x^{\sigma_0 + \eps})$ for some $a>0$ and any $\eps > 0$. Then, applying successively an elementary Tauberian argument yields $N_{m-1}(x)=ax+O_{\eps}(x^{1-\frac{1-\sigma_0}{2} + \eps}),\ldots, N(x)= ax+O_{\eps}(x^{1-\frac{1-\sigma_0}{2^m} + \eps})$ and the integers $\mathcal{N}$ are well-behaved.

From the $[0,1]$-system of \cite{BDV2022} one can create a new system $(\mathcal{P}_{\alpha}, \mathcal{N}_{\alpha})$ as $\mathcal{P}_{\alpha} = \mathcal{P} \cup \{p^{1/\alpha}| p \in\mathbb{P} \}$. The primes of this new system are $\alpha$-well-behaved in view of $\pi_{\mathcal{P}_{\alpha}}(x) =\pi_{\MP}(x) + x^{\alpha}/(\alpha \log x) + O(x^\alpha/ \log^2 x)$. The zeta function $\zeta_{\mathcal{P}_{\alpha}}(s) = \zeta_{\MP}(s)\zeta(s/\alpha)$ remains not polynomially bounded (away from the pole at $s = 1$) on any half-plane $\Re s \geq \sigma_{0}$, $\sigma_0 < 1$, inheriting this property from the zeta function of $\mathcal{P}$, as $1/\zeta(s/\alpha)$ is bounded on every half-plane $\Re s \geq 1 - \eps$ with $\eps<1-\alpha$. Therefore the integers $\mathcal{N}_{\alpha}$ cannot be well-behaved as that would entail polynomial bounds for $\zeta_{\MP_{\alpha}}(s) - a/(s-1)$ in a half-plane $\Re s \geq 1-\eps$ via the same argument preceding this remark.
\end{remark}

\section{Systems with $\beta < 1/2$} \label{sec:betalesshalf}

The construction of Beurling prime number systems having \emph{very well-behaved integers} in the sense that $\beta<1/2$ presents a number of new challenges. The most glaring one is the failure of the strategy of constructing first \emph{continuous} Beurling systems with the required properties and then generating a discrete example through one of the currently available random approximation procedures. Namely, a direct application of Theorem \ref{th: discretization} will always yield an error term of at least $O(x^{1/2})$ on the integers. So, it appears that considering continuous Beurling systems is of little help here, and one has to construct a discrete example right from the start.

As the integers have to display the best behavior, it seems natural to define a Beurling system through the sequence of the integers rather than the primes. Yet, it appears that for sequences defined through this philosophy, it is often extremely difficult to show which behavior the primes must admit. For example, for the classical primes (or for the number of prime ideals in the ring of integers of an algebraic number field), we cannot yet prove $\alpha < 1$.

As such, throughout this section we will assume the \emph{Riemann hypothesis}, asserting that $(\mathbb{P},\mathbb{N})$, the classical primes and integers, are a $[1/2, 0]$-system. On this template system, we shall consider a few perturbations providing $[\alpha,\beta]$-systems for other values of $\alpha,\beta$. As mentioned in Theorem \ref{wbs: thrh} in the introduction, we shall show the existence of $[1/2, \beta]$-systems for each $0 \le \beta<1/2$, and of $[\alpha, \beta]$-systems for each $\alpha$, $\beta$ with $1/2 < \alpha < 2/3$ and $2\alpha/(\alpha + 2) \leq \beta < 1/2$.

We first consider the simplest case when $\alpha = 1/2$. We shall add a few primes to the classical primes in order to generate an extra term in the integer counting function. Specifically, for $0 < \beta < 1/2$, we consider the Beurling prime number system $\mathcal{P}_{\beta} = \mathbb{P} \cup \mathbb{P}^{1/\beta} = \{p,p^{1/\beta}\mid p \text{ prime}\}$. We immediately obtain (assuming RH) that\footnote{To simplify the notation, we write in this section $\pi_{\beta}$ instead of $\pi_{\MP_{\beta}}$ etc.\ for the counting functions associated with $\MP_{\beta}$.} $\pi_{\beta}(x) = \pi(x) + \pi(x^{\beta}) = \Li(x) + O_{\eps}(x^{1/2 + \eps})$ for each $\eps > 0$. The error term does not hold for any $\eps < 0$ as can be seen from the presence of $\zeta$-zeros on the line  $\Re s = 1/2$ or from \cite[Th.\ 1]{Hilberdink2005} after deducing the behavior of the integers. Alternatively one may also apply Littlewood's \cite{Littlewood1914} well-known oscillation result $\pi(x) =\Li(x) + \Omega(\sqrt{x} \log\log\log x /\log x)$. It remains to analyze the behavior of the integers. The integers of this system are formed by the products $nl^{1/\beta}$ where $n,l$ represent classical integers. We shall apply Dirichlet's hyperbola method to get an asymptotic formula for $N_{\beta}(x)$. For later use we show something a little bit more general in the following lemma.

\begin{lemma} \label{wbs: lemdirichlet} Let $\mathcal{N},\mathcal{L}$ be subsets\footnote{These are allowed to be multisets.} of $[1,\infty)$, $h: \mathcal{L} \rightarrow \mathbb{R}^{+}$, $a,b \geq 0$, $0 \leq \gamma, \delta < \beta < 1$. If
\begin{align}
 N(x) \coloneqq \sum_{n \leq x, n \in \mathcal{N}} 1 & = ax + O(x^{\gamma}), \\
 L(x) \coloneqq \sum_{l \leq x, l \in \mathcal{L}} h(l) &= bx^{\beta} + O(x^{\delta}),
\end{align}
then
\[
	\sum_{\substack{nl \leq x \\n \in \mathcal{N}, l\in \mathcal{L}}} h(l) = aH(1)x + bI(\beta)x^{\beta} + O\left(x^{\frac{\beta - \gamma \delta}{1-\gamma + \beta - \delta}}\right),
\]
where $H(1) = \sum_{l \in \mathcal{L}} h(l)/l$ and $I(\beta) = \lim_{R \rightarrow \infty} \bigg( \sum_{n \leq R, n \in \mathcal{N}}n^{-\beta} -aR^{1-\beta}/(1-\beta) \bigg)$, that is the value at $\beta$ of the meromorphic extension of $\sum_{n \in \mathcal{N}} n^{-s}$.
\end{lemma}
\begin{proof} In order to avoid an overload of notation we shall omit the subscripts $n \in \mathcal{N}$ and $l \in \mathcal{L}$ in the sums throughout the proof. Let $1 \leq y \leq x$. Applying Dirichlet's hyperbola method gives
\[
	  \sum_{nl \leq x} h(l) = \sum_{n \leq y} L(x/n) + \sum_{l \leq x/y} h(l) N(x/l) - N(y) L(x/y).
\]
The first term delivers
\[
	\sum_{n \leq y} L(x/n) = bx^{\beta} \sum_{n \leq y} n^{-\beta} + x^{\delta} \sum_{n \leq y}O(n^{-\delta}) = bx^{\beta}I(\beta) + \frac{abx^{\beta}y^{1-\beta}}{1-\beta} + O(x^{\beta}y^{\gamma - \beta}) + O(x^{\delta} y^{1-\delta}).
\]
The second term gives
\[
	ax\sum_{l \leq x/y} h(l)/l + x^{\gamma}\sum_{l \leq x/y} h(l)O(l^{-\gamma}) = axH(1) - \frac{\beta ab x^{\beta}y^{1-\beta }}{1-\beta} + O(x^{\delta}y^{1 - \delta} ) + O(x^{\beta} y ^{\gamma - \beta}).
\]
The third term, finally, is
\[
	(ay + O(y^{\gamma}))(bx^{\beta}y^{-\beta} + O(x^{\delta}y^{-\delta})) = abx^{\beta}y^{1-\beta} + O(x^{\beta}y^{\gamma - \beta}) + O(x^{\delta}y^{1-\delta}).
\]
The result follows after optimizing $y = x^{\frac{\beta - \delta}{1 - \gamma + \beta - \delta}}$.
\end{proof}
Applying this lemma with $\mathcal{N} = \mathbb{N}$, $\mathcal{L} = \mathbb{N}^{1/\beta}$, $a=b =1$, $h \equiv 1$ and $\gamma = \delta = 0$, delivers
\[
	N_{\beta}(x) = \zeta(1/\beta) x + \zeta(\beta) x^{\beta} + O(x^{\frac{\beta}{\beta + 1}}),
\]
so the generalized integers are indeed $\beta$-well-behaved and not better (as $\zeta$ does not vanish on the positive real line). This concludes the analysis for systems with $\alpha = 1/2$.

\medskip
In order to find systems with $\alpha > 1/2$, we must introduce a deviation on the primes without completely destroying the behavior of the integers. This cannot be achieved by simply adding new primes to the system. Namely, considering a system of the form $\mathcal{P} = \mathbb{P} \cup \mathcal{A}$ results in a perturbation of $\log \zeta(s)$, that is $\log \zeta_{\mathcal{P}}(s) = \log \zeta(s) + \log\zeta_{\mathcal{A}}(s)$, where $\log\zeta_{\mathcal{A}}(s)$ must, in view of Landau's theorem \cite[Th.\ 1.7]{MontgomeryVaughanbook}, admit a singularity at $s =\alpha$, in order for $\mathcal{P}$ to be an $\alpha$-system as $\Pi_{\mathcal{A}}(x)$ must then be both $O(x^{\alpha + \eps})$ and $\Omega(x^{\alpha - \eps})$. On the other hand, if the integers $\mathcal{N}$ would be $\beta$-well-behaved with $\beta < \alpha$, then $\zeta_{\mathcal{P}}(s) = \zeta(s) \zeta_{\mathcal{A}}(s)$ and thus also $\zeta_{\mathcal{A}}(s)$ would be analytic at $s =\alpha$. Now consider the limit as $\sigma \rightarrow \alpha +$ of the positive function $\log \zeta_{\mathcal{A}}(\sigma)$. This limit clearly exists (and may be $\infty$) as $\log \zeta_{\mathcal{A}}(\sigma)$ is decreasing. If it were $\infty$, then also $\zeta_{\mathcal{A}} (\sigma) = \exp(\log \zeta_{\mathcal{A}}(\sigma))$ must tend to $\infty$ as $\sigma \rightarrow \alpha+$, contradicting that $\zeta_{\mathcal{A}}(s)$ is analytic at $s = \alpha$. On the other hand, if the limit were a finite nonnegative number, this would imply that $\zeta_{\mathcal{A}}(\alpha) \geq 1$. Now, if $\zeta_{\mathcal{A}}(s)$ is analytic at $s =\alpha$ and as $\zeta_{\mathcal{A}}(\alpha) \neq 0$, it follows that the branch of the logarithm of $\zeta_{\mathcal{A}}(s)$ extending $\log \zeta_{\mathcal{A}}(s)$ on the half-plane $\Re s > \alpha$ is also analytic at $s = \alpha$, contradicting the fact that $\log \zeta_{\mathcal{A}}(s)$ would have a singularity there. In conclusion, one cannot form $[\alpha,\beta]$-systems with $\alpha > 1/2$ and $\beta < 1/2$ by simply adding extra generalized primes to the classical primes $\mathbb{P}$ (if RH is true).

The above obstruction does not occur when the primes are subtracted instead as the singularity of $-\log \zeta_{\mathcal{A}}(s)$ at $s = \alpha$ may turn into a zero after exponentiation. However, we cannot subtract an arbitrary sequence; we can only delete the primes that are already there! Our strategy will therefore be to first find a subsequence of the primes that we can delete whose counting function is asymptotically about $\Li(x^\alpha)$ and whose perturbation to $\zeta(s)$ has good analytic behavior. For this we shall once again use Theorem \ref{th: discretization}.

We consider the function
\begin{equation}
\label{def F}
	F(x) = \int^{x}_{1} u^{\alpha-1} \dif \:(\pi(u) + E(u))  = \sum_{\substack{p\le x\\p\in\mathbb{P}}} p^{\alpha-1} + \int^{x}_{1} u^{\alpha-1}\dif E(u),
\end{equation}
where $\dif E$ is an error measure supported on $\mathbb{P}$ that we shall define in a short while.
We briefly recall the construction of the sequence $p_j$ from the proof of Theorem \ref{th: discretization} for purely discrete measures $\dif F(x)$. One considers independent random variables $P_j$ whose (cumulative) distribution function is given by
\[
P(P_j \leq x) = \begin{cases}
 	0 & \text{if } F(x) \leq j-1, \\
 	F(x) - j+1& \text{if } j-1 < F(x) < j, \\
 	1 & \text{if } F(x) \geq j.
\end{cases}
\]
In the proof one then shows that the desired property \eqref{eq: bound exp sums} holds almost surely (upon replacing $p_j$ with $P_j$). In particular, there must therefore exist a sequence $p_j$ for which the requirements of Theorem \ref{th: discretization} are met, as by definition of the random variables $P_{j}$, the counting function of the sequence $p_j$ is at most $1$ apart from $F$. As the measure $\dif F$ is supported on $\mathbb{P}$, so are the probability measures associated with the random variables $P_{j}$, so that the elements $p_j$ must in our case all be classical primes.

The purpose of the error measure $\dif E$ is to ensure that the primes selected by Theorem \ref{th: discretization} are all \emph{different} as we cannot delete a prime twice from our system. One may continue without this error, but then the construction and analysis of the Beurling system becomes more complicated, see Remark \ref{rem: prob}. Let $q_j$ be the first classical prime for which $ \int^{q_{j}}_{1} u^{\alpha-1} \dif \pi(u) \geq j$. We define $\dif E$ as the measure transferring the \emph{excess probability} of $q_j$ to the next classical prime, that is $\dif E = \sum_{j= 1}^{\infty} \dif E_j$ where $\dif E_j\{p\} = 0 $ for each $p \in \mathbb{P}$ except at $\dif E_j\{q_j\} = j - \int^{q_{j}}_{1} u^{\alpha-1} \dif \pi(u) $ and $\dif E_j\{q_{j}^{+}\}  = -j + \int^{q_{j}}_{1} u^{\alpha-1} \dif \pi(u)$. Here $q_{j}^{+}$ denotes the next classical prime after $q_j$. 
We have, denoting the classical prime preceding $q_{j}$ by $q_{j}^{-}$,
\[
	\int_{1}^{x}|\dif E(u)| \le \sum_{q_{j}\le x}2\biggl(\int_{1}^{q_{j}}u^{\alpha-1}\dif\pi(u)-j\biggr) \le \sum_{q_{j}\le x} 2\biggl(\sum_{p\le q_{j}}p^{\alpha-1}-\sum_{p\le q_{j}^{-}}p^{\alpha-1}\biggr)
	 = 2\sum_{q_{j}\le x}q_{j}^{\alpha-1},
\]	
as $\sum_{p\le q_{j}^{-}}p^{\alpha-1} < j$ by definition of $q_{j}$. Therefore, putting $Q(x) \coloneqq \sum_{q_j \le x} 1$ for the counting function of the $q_j$, we obtain
\begin{equation}
\label{bound dif E}
	\int_{1}^{x}|\dif E(u)| \le \sum_{q_{j}\le x}2q_{j}^{\alpha-1} = 2\int_{1}^{x} u^{\alpha-1} \dif Q(u) 
	\ll \frac{x^{2\alpha-1}}{\log x},
\end{equation}
where in the last estimate we used partial summation and the fact that $Q(x) =  \int_{1}^{x}u^{\alpha-1}\dif\pi(u)+O(1) \sim x^{\alpha}/(\alpha\log x)$ by the definition of the sequence $q_j$ and the prime number theorem.

Let us denote $\mathcal{P}_\mathcal{S} = \{p_j: j\ge1\}$ for the Beurling system generated by the probabilistic procedure from Theorem \ref{th: discretization} applied to $F$. Subtracting the elements of $\mathcal{P}_\mathcal{S}$ from $\mathbb{P}$, the characteristic function of the remaining integers $\MN =  \{n \in \mathbb{N}: p | n \Rightarrow p \notin \mathcal{P}_{\mathcal{S}}\}$ can be written as the convolution $1_{\MN} = 1_{\mathbb{N}}\ast\mu_{\mathcal{S}}$, where $1_{\mathcal{A}}$ denotes the characteristic function of a set $\mathcal{A}$, and $\mu_{\mathcal{S}}$ is the M\"obius function of $\mathcal{P}_\mathcal{S}$, that is, the multiplicative function defined by $\mu_{\mathcal{S}}(p_{j}) = -1$, $\mu_{\mathcal{S}}(p_{j}^{\nu}) = 0$ for $\nu\ge2$, and $\mu_{\mathcal{S}}(p^{\nu}) = 0$, for $p\notin \MP_{\mathcal{S}}$, $\nu\ge1$.
Therefore, we wish to find a strong estimate on $M_{\mathcal{S}}(x) \coloneqq \sum_{ n \leq x} \mu_{\mathcal{S}}(n)$.

In view of \eqref{def F} and \eqref{bound dif E}, the prime number theorem yields
\[
	F(x) = \Li(x^{\alpha})\bigl(1+o(1)\bigr) + O\biggl(\frac{x^{3\alpha-2}}{\log x}\biggr) \sim \Li(x^{\alpha}).
\]
As $F(x)$ is now about $x^{\alpha}/ \log x$, the estimate \eqref{eq: bound exp sums} actually holds with $x^{\alpha/2}$ replacing $\sqrt{x}$, that is,
\begin{equation} \label{eq: adjusted}
 \sum_{p_{j}\le x}p_{j}^{-\I t} - \int_{1}^{x}u^{-\I t}\dif F(u) \ll x^{\alpha/2} + x^{\alpha/2}\sqrt{\frac{\log(\,\abs{t}+1)}{\log(x+1)}}.
\end{equation}

One may for instance apply Theorem \ref{th: discretization} to $\tilde{F}(x) \coloneqq F(x^{1/\alpha})$ instead, see also \cite[Remark 2.3]{BrouckeVindas}.
Via integration by parts, the fact that $\Pi_{\mathcal{S}}(x) = \pi_{\mathcal{S}}(x) + O(x^{\alpha/2})$, and the adjusted estimate \eqref{eq: adjusted}, one obtains that on each half-plane $\Re s \geq \alpha/2 + \eps$,
\begin{align*}
	\log \zeta_{\mathcal{S}}(s)
		& = \int^{\infty}_{1} u^{-s} \dif \Pi_{\mathcal{S}}(u) = \int^{\infty}_{1} u^{-s} \dif \pi_{\mathcal{S}}(u) + O_{\eps}(1) \\
		& = \int^{\infty}_{1} u^{-\sigma}\dif\,\biggl(\sum_{p_j \leq u} p_j^{-\I t}\biggr) + O_{\eps}(1) = \int^{\infty}_{1} u^{-s} \dif F(u) + O_{\eps}(\sqrt{\log|t|}) \\
		&= \int^{\infty}_{1} u^{-s+\alpha - 1}\dif\Pi(u) + O\left(\int^{\infty}_1 u^{-\sigma}\abs{\dif E(u)}\right) + O_{\eps}(\sqrt{\log|t|}) \\
		& = \log \zeta(s + 1-\alpha) + O_{\eps}(\sqrt{\log|t|}).
\end{align*}
Note that we use here that $\alpha < 2/3$ to ensure that the integral involving $\abs{\dif E}$, which converges on $\Re s > 2\alpha - 1$ in view of \eqref{bound dif E}, also converges on $\Re s \geq \alpha/2 + \eps$. This implies that both $\zeta_{\mathcal{S}}(s)$ and $1/\zeta_{\mathcal{S}}(s)$ are $\ll_\eps |t|^{\eps}$ on each half-plane $\Re s \geq \alpha/2 + \eps$. (RH implies that both $\zeta(s)$ and $1/\zeta(s)$ are $\ll_\eps |t|^{\eps}$ on half-planes $\Re s \geq 1/2 + \eps$, see for instance \cite[Th.\ 13.18 and Th.\ 13.23]{MontgomeryVaughanbook}.) A small adaptation of the Perron inversion argument of Section \ref{sec:alphalessbeta} from \eqref{eq: secondPerron} onwards where the location of the pole is now moved to $\alpha$ and the bounds on the half-plane $\Re s > \alpha /2 + \eps$ for $|\zeta(s)|$ are now $O_{\eps}(|t|^{\eps})$ delivers that $N_{\mathcal{S}}(x) = ax^{\alpha} + O_\eps(x^{\alpha/2 + \eps})$, where $a =  \res_{s=\alpha}\zeta_{\mathcal{S}}(s)/\alpha$. The exact same inversion argument applied to the non-decreasing function $N_{\mathcal{S}}(x) + M_{\mathcal{S}}(x)$ whose Mellin--Stieltjes transform is $\zeta_{\mathcal{S}}(s) + 1/\zeta_{\mathcal{S}}(s)$ then gives $N_{\mathcal{S}}(x) + M_{\mathcal{S}}(x) = ax^{\alpha} + O_\eps(x^{\alpha/2 + \eps})$. Therefore, we have for each $\eps > 0$,
\[
	M_{\mathcal{S}}(x) = \sum_{ n \leq x} \mu_{\mathcal{S}}(n) \ll_\eps x^{\alpha/2 + \eps}.
\]

We are now ready to define our Beurling system. Our system $\mathcal{P}_{\alpha,\beta}$ is given by $\mathcal{P}_{\alpha,\beta} =  \mathbb{P} \cup \mathbb{P}^{1/\beta} \setminus \mathcal{P}_\mathcal{S}$, with $\pi_{\alpha, \beta}(x) = \pi(x) + \pi(x^{\beta}) -\pi_{\mathcal{S}}(x)$. Recalling that $\beta < 1/2 < \alpha$, RH and the asymptotic $\pi_{\mathcal{S}}(x) = F(x)+O(1) \sim \Li(x^{\alpha})$ immediately imply that $\mathcal{P}_{\alpha,\beta}$ is an $\alpha$-system for the primes. It remains to analyze the integers. One has $\dif N_{\alpha,\beta}(u) = \dif \,\lfloor u\rfloor \ast \dif\, \lfloor u^{\beta}\rfloor \ast \dif M_{\mathcal{S}}(u)$. We begin by analyzing the convolution of the classical integers with the M\"obius function from the Beurling system $(\mathcal{P}_\mathcal{S},\mathcal{N}_\mathcal{S})$. Applying Lemma \ref{wbs: lemdirichlet} with the nonnegative functions
$h(l) = \mu_{\mathcal{S}}(l) + 1_{\mathcal{S}}(l)$ and $h(l) = 1_{\mathcal{S}}(l)$ gives, after subtracting them from one another,

\[
	\sum_{\substack{nl \leq x \\ n \in \mathbb{N}, l \in \mathcal{N}_\mathcal{S}}} \mu_{\mathcal{S}}(l) = \frac{x}{\zeta_{\mathcal{S}}(1)} + O_{\eps}\left(x^{\frac{2\alpha}{\alpha + 2} + \eps}\right).
\]
We now apply Lemma \ref{wbs: lemdirichlet} once more with $\mathcal{N} = \{n \in \mathbb{N}: p | n \Rightarrow p \notin \mathcal{P}_{\mathcal{S}}\}$, $\mathcal{L} = \mathbb{N}^{1/\beta}$ and $h = 1_{\mathcal{L}}$ to find the estimate on the last convolution;
\[
	N_{\alpha,\beta}(x) = \sum_{\substack{nlm \leq x \\ n \in \mathbb{N}, l \in \mathcal{N}_\mathcal{S} \\  m \in \mathbb{N}^{1/\beta}}} \mu_{\mathcal{S}}(l)
	= \frac{\zeta(1/\beta)  }{\zeta_{\mathcal{S}}(1)} x + \frac{\zeta(\beta) }{\zeta_{\mathcal{S}}(\beta)} x^{\beta} + O_\eps\left(x^{\frac{\beta}{1 + \beta - 2\alpha/(\alpha + 2)}+ \eps}\right),
\]
where we have used that $\beta > 2\alpha/(\alpha + 2)$. Note that $\beta > \alpha/2$, so $\zeta_{\mathcal S}(s)$ cannot vanish at $s =1$ or $s=\beta$ as $\log \zeta_{\mathcal{S}}(\sigma)  = \log \zeta(\sigma+1-\alpha) + O_\eps(1)$ for $\sigma > \alpha/2 + \eps$. So, this estimate implies that the integers of the system $\mathcal{P}_{\alpha,\beta}$ are $\beta$-well-behaved (and not better). This concludes the proof of Theorem \ref{wbs: thrh} in the case $\beta > 2\alpha/(\alpha + 2)$.

In case $\beta = 2\alpha/(\alpha + 2)$ the only difference is that we can no longer apply Lemma \ref{wbs: lemdirichlet} to obtain the estimate for the final convolution of $\mathcal{N}$ with $\mathbb{N}^{1/\beta}$. However, in that case the estimate can be obtained by even simpler means, that is, by writing
\[
	N_{\alpha,\beta}(x) = \sum_{\substack{m \leq x \\ m \in \mathbb{N}^{1/\beta}}} \biggl(\frac{x}{m \zeta_{\mathcal{S}}(1)} + O_\eps\Bigl(\frac{x^{\beta + \eps}}{m^{\beta + \eps}}\Bigr)\biggr).
\]
This gives $N_{\alpha,\beta}(x) =  \zeta(1/\beta)x /\zeta_{\mathcal{S}}(1) + O_{\eps}(x^{\beta + \eps})$. The integers of the system $\mathcal{N}_{\alpha,\beta}$ can also not be better than $\beta$-well-behaved as that would entail that the zeta function $\zeta_{\alpha,\beta}(s) = \zeta(s)\zeta(s/\beta)/\zeta_{\mathcal{S}}(s)$ is analytic at $s =\beta$, which is false. This concludes the proof of Theorem \ref{wbs: thrh} in all cases.

\begin{remark} \label{rem: prob}
As we already mentioned in the proof, one may also work without the error measure $\dif E$ in the definition of $F(x)$. So, upon setting $\dif E \equiv 0$, it is not excluded that an element occurs twice in the sequence $p_j$. This may happen if $F(p-) < j$ and $F(p) > j$ as then the random variables $P_j$ and $P_{j+1}$ both have a nonzero probability to equal $p$. We describe here how to proceed in these circumstances.

We require an estimate on the number of these double primes. In our final system we will add these double primes once again; after all we need those primes an additional time to be able to delete them twice! We shall show that
\begin{equation} \label{eq: wbsdoubleprimes} \sum_{\substack{P_{j} \leq x \\ P_j = P_{j+1}}} 1 = O(x^{\alpha/2})
\end{equation}
holds with probability $1$ and therefore there must be a sequence $p_j$ that next to the requirements of Theorem \ref{th: discretization} additionally satisfies $\pi_{\mathcal{D}}(x) = O(x^{\alpha/2})$, where $\mathcal{P}_{\mathcal{D}}$ consists\footnote{The set $\mathcal{P}_{\mathcal{D}}$ may technically not be a Beurling system as it is possible $\mathcal{P}_\mathcal{D}$ only contains finitely many elements or it may even be empty.} of the elements of $\mathcal{P}_{\mathcal{S}}$ occurring twice. Let $D_j$ denote the random variable which is $1$ if $P_j = P_{j+1}$ and $0$ otherwise and let again $q_j$ denote the prime where $F(q_j -) < j$ and $F(q_j) \geq j$. We analyze the sum $\sum_{j = 1}^{J} D_j$. The random variables $(D_j)_{j\ge1}$ are not independent, but the subsequences $(D_{2j})_{j\ge1}$ and $(D_{2j-1})_{j\ge1}$ with even and odd indices both only consist of independent random variables. We have $P(D_{j}=1)=P(P_{j}=q_{j}=P_{j+1}) \le q_j^{2\alpha-2}/4$, as $\max_{a+b=\lambda} ab = \lambda^{2}/4$.

 Therefore, if we manage to show that
\[
	B = B(J) =  \sum_{j = 1}^{J} B_j \ll q_J^{\alpha/2}, \quad J \rightarrow \infty,
\]
holds with probability $1$, where $B_j$ are independent Bernoulli variables equaling $1$ with probability $q_j^{2\alpha-2}/4$, then also $\sum_{j = 1}^J D_j \ll q_J^{\alpha/2}$ holds almost surely and this would imply that \eqref{eq: wbsdoubleprimes} holds almost surely as well.

Observe that
\[
	J = \lfloor F(q_J)\rfloor \leq \int^{q_J}_{1} \dif F(u) = \int^{q_J}_{1}  u^{\alpha-1}\dif\pi(u) \asymp \frac{q_J^{\alpha}}{\log q_J}.
\]
An application of a probabilistic lemma for estimating sums of independent random variables such as Hoeffding's inequality \cite[Th.\ 1]{Hoeffding} gives
\[
	P\left(B - \mathbb{E}(B) \geq \sqrt{J\log J}\right) \leq \exp(-2 \log J) = J^{-2}.
\]
Here $\mathbb{E}(B)$ stands for the expected value of $B$ and satisfies
\[
	\mathbb{E}(B) = \sum_{j = 1}^J \frac{q_{j}^{2\alpha-2}}{4} \leq \int^{q_J}_{1} \frac{u^{-2(1-\alpha)}}{4} \dif F(u) = \int^{q_J}_{1} \frac{u^{-3(1-\alpha)}}{4}\dif \pi(u) \ll \frac{q_J^{3\alpha - 2}}{\log q_J}.
\]
We therefore have that $P(B \geq cq_J^{\alpha/2}) \leq J^{-2}$ for a suitable absolute constant $c > 0$. Note that we use here that $\alpha < 4/5$ to guarantee that $q_J^{3\alpha - 2} \ll q_J^{\alpha /2}$. As $J^{-2}$ is summable, the Borel--Cantelli lemma implies that for each $\eps > 0$ there exists $J_0$ such that $P(B \geq cq_J^{\alpha/2}) \leq \eps$ for each $J \geq J_0$. In other words, $B \ll q_J^{\alpha/2}$ as $J \rightarrow \infty$ almost surely and this concludes the analysis of the double primes in the system $\mathcal{P}_{\mathcal{S}}$.

Summarizing, we have shown above the existence of a sequence of primes $\mathcal{P}_{\mathcal{S}}$ where each prime occurs at most twice with the following properties. Denoting $\mathcal{P}_{\mathcal{D}}$ as the set comprising the primes that occur twice in $\mathcal{P}_{\mathcal{S}}$, we have shown that the arithmetic functions of these systems satisfy, for each $\eps > 0$,
\begin{gather*}
	\pi_{\mathcal{S}}(x)  \asymp \frac{x^{\alpha}}{\log x},  \quad N_{\mathcal{S}}(x)  = ax^{\alpha} + O_{\eps}(x^{\alpha/2 + \eps}),
	\quad M_{\mathcal{S}}(x)  = \sum_{\substack{n \leq x\\ n \in \mathcal{N}_\mathcal{S}}} \mu_\mathcal{S}(n) \ll_\eps x^{\alpha/2 + \eps}, \\
 	\pi_{\mathcal{D}}(x)  = \sum_{\substack{p \leq x \\ p \in \mathcal{P}_\mathcal{D}}} 1 \ll x^{\alpha/2}, \quad N_{\mathcal{D}}(x)  \ll_\eps x^{\alpha/2 + \eps}.
\end{gather*}
The estimate on $N_{\mathcal{D}}(x)$ can be deduced from
\[
	N_\mathcal{D}(x) \leq x^{\alpha/2 + \eps}\int^{x}_{1^{-}} u^{-\alpha/2 - \eps}\dif N_\mathcal{D}(u)
	 \leq x^{\alpha/2 + \eps}\int^{\infty}_{1^{-}} u^{-\alpha/2 - \eps}\dif N_\mathcal{D}(u)  \ll_\eps x^{\alpha/2 + \eps},
\]
exploiting the boundedness of $\zeta_\mathcal{D}(s)$ on the half-planes $\Re s \geq \alpha/2 + \eps$, which is guaranteed by the estimate on $\pi_{\mathcal{D}}$.

Our Beurling system is now defined as $\mathcal{P}_{\alpha,\beta} =  \mathbb{P} \cup \mathcal{P}_\mathcal{D} \cup \mathbb{P}^{1/\beta} \setminus \mathcal{P}_\mathcal{S}$ (interpreted as multisets). With multiple applications of Lemma \ref{wbs: lemdirichlet}, one can then, similarly as before, establish
\[
	N_{\alpha,\beta}(x) = \sum_{\substack{nlrm \leq x \\ n \in \mathbb{N}, l \in \mathcal{N}_\mathcal{S} \\ r \in \mathcal{N}_\mathcal{D}, m \in \mathbb{N}^{1/\beta}}} \mu_{\mathcal{S}}(l)
	= \frac{\zeta_{\mathcal{D}}(1) \zeta(1/\beta)  }{\zeta_{\mathcal{S}}(1)} x + \frac{\zeta_{\mathcal{D}}(\beta)\zeta(\beta) }{\zeta_{\mathcal{S}}(\beta)} x^{\beta} + O_\eps\left(x^{\frac{\beta}{1 + \beta - 2\alpha/(\alpha + 2)}+ \eps}\right),
\]
for $\beta > 2\alpha/(\alpha + 2)$.

\end{remark}

\begin{remark} \:
\begin{enumerate}
\item The construction considering double primes is more precise than the one involving the error measure. In the approach of the double primes, the only restriction on $\alpha$ was  $\alpha < 4/5$, whereas it was $\alpha < 2/3$ when involving $\dif E$. However, if $2/3 \leq \alpha < 4/5$, then $\beta$ needs to be larger than $1/2$ as $\beta \geq 2\alpha/(\alpha + 2) \geq 1/2$. This then does not generate new $[\alpha,\beta]$-systems as systems with those parameters were already constructed in Section \ref{sec:betalessalpha} without assuming RH.

\item We do not exclude the possibility that certain error terms in the calculation of $N_{\alpha,\beta}(x)$ can be improved with more advanced technology than Dirichlet's hyperbola method. Such improvements, possibly combined with the double primes approach, might give rise to a larger region of acceptable $\alpha$ and $\beta$ in Theorem \ref{wbs: thrh}.
\item Algebraic number fields are other natural examples of Beurling systems providing at least conjecturally very well-behaved integers. We include here what is currently known and conjectured about the values $\alpha$ and $\beta$ for a number field $K$ of degree $n$. It was shown by Landau (see e.g.\ \cite{Landau}) that the number of integral ideals $N_{K}(x)$ of norm at most $x$ satisfies $N_{K}(x) = a_{K}x + O_{K}(x^{1-\frac{2}{n+1}})$, improving a result by Weber, and that the number of prime ideals $\pi_{K}(x)$ of norm at most $x$ satisfies $\pi_{K}(x) = \Li(x) + O_{K}\bigl(x\exp(-c_{K}\sqrt{\log x})\bigr)$, for certain positive constants $a_{K}$ and $c_{K}$ depending on the number field. Landau also showed that $N_{K}(x) = a_{K}x + \Omega_{K}(x^{\frac{1}{2}-\frac{1}{2n}})$. The exponent $1-2/(n+1)$ in Landau's ideal theorem was improved in the quadratic case by Huxley and Watt \cite{HuxleyWatt}, in the cubic case by M\"uller \cite{Muller}, and in the case $n\ge4$ by Nowak \cite{Nowak}. (Note that the proof of a claimed improvement \cite{Bordelles} of Nowak's result turned out to be flawed \cite{Bordellescorrigendum}.) It is conjectured that $N_{K}(x) = a_{K}x + O_{K,\eps}(x^{\frac{1}{2}-\frac{1}{2n}+\eps})$ for all $\eps>0$. The conjectural extended Riemann hypothesis (ERH) on the zeros of Dedekind zeta functions implies that $\pi_{K}(x) = \Li(x) + O_{K,\eps}(x^{1/2+\eps})$ for every $\eps>0$. Hence we can summarize that the Beurling number systems formed by the (norms of the) prime and integral ideals in a number field $K$ of degree $n$ are conjectured to be $[1/2, 1/2-1/(2n)]$-systems.

Naturally one can also perturb these systems as we did with the system $(\mathbb{P}, \N)$, but they give a smaller range of $[\alpha,\beta]$-systems.

\end{enumerate}
\end{remark}

\end{document}